\documentclass[12pt]{article}
\usepackage[T1]{fontenc}
\usepackage{amsmath,amsfonts,amssymb,amsthm}
\usepackage{bm}
\usepackage{tcolorbox}
\definecolor{pantone312}{HTML}{009DD1}
\usepackage{mathtools}
\usepackage{bbm}
\usepackage{cleveref}
\usepackage{enumerate}
\usepackage{esint}
\usepackage[mathscr]{euscript}
\let\euscr\mathscr \let\mathscr\relax
\usepackage{graphicx}
\usepackage{tabularx}
\usepackage{color}
\newcommand{\avint}{\int \hspace{-1.0em}-\,}

\definecolor{darkgreen}{rgb}{0,0.55,0}
\newcommand{\N}{\mathbbm{N}}
\newcommand{\cs}{\c{s}}
\newcommand{\R}{\mathbbm{R}}

\newcommand{\g}{\mbox{\sl g}}

\newcommand{\F}{{\mathcal F}}
\newcommand{\CalA}{\mathcal{A}}
\newcommand{\CalB}{\mathcal{B}}
\newcommand{\CalC}{\mathcal{C}}

\newcommand{\CalH}{\mathcal{H}}
\newcommand{\CalJ}{\mathcal{J}}
\newcommand{\CalK}{\mathcal{K}}
\newcommand{\CalG}{\mathcal{G}}
\newcommand{\CalL}{\mathcal{L}}
\newcommand{\CalM}{\mathcal{M}}
\newcommand{\CalU}{\mathcal{U}}
\newcommand{\CalO}{\mathcal{O}}
\newcommand{\CalT}{\mathcal{T}}
\newcommand{\CalV}{\mathcal{V}}
\newcommand{\corrM}{\bm{\nu}}
\newcommand{\T}{\mathbbm{T}}

\DeclareMathOperator{\supp}{supp}
\DeclareMathOperator{\diverg}{div}
\DeclareMathOperator*{\esssup}{ess\,sup}

\usepackage{sectsty}

\sectionfont{\large}
\subsectionfont{\normalsize}
\subsubsectionfont{\normalsize}
\paragraphfont{\normalsize}

\newcommand{\dd}{\,{\rm d}}
\newcommand{\dx}{\,\dd x}
\newcommand{\dt}{\, \dd t}

\def\Xint#1{\mathchoice
{\XXint\displaystyle\textstyle{#1}}%
{\XXint\textstyle\scriptstyle{#1}}%
{\XXint\scriptstyle\scriptscriptstyle{#1}}%
{\XXint\scriptscriptstyle\scriptscriptstyle{#1}}%
\!\int}
\def\XXint#1#2#3{{\setbox0=\hbox{$#1{#2#3}{\int}$ }
\vcenter{\hbox{$#2#3$ }}\kern-.59\wd0}}
\def\avint{\Xint-}

\newtheorem{prop}{Proposition}
\newtheorem{theorem}{Theorem}
\newtheorem{lemma}{Lemma}
\newtheorem{remark}{Remark}
\newtheorem{definition}{Definition}

\allowdisplaybreaks

\begin{document}
\phantom{ }
\vspace{4em}

\begin{flushleft}
{\large \bf Probabilistic Descriptions of Fluid Flow: A Survey}\\[2em]
{\normalsize \bf Dennis Gallenm\"uller}\\[0.5em]
\small Institute of Applied Analysis, Ulm University, Germany.\\
E-mail: dennis.gallenmueller@alumni.uni-ulm.de\\[3em]
{\normalsize \bf Raphael Wagner}\\[0.5em]
\small Institute of Applied Analysis, Ulm University, Germany.\\
E-mail: raphael.wagner@uni-ulm.de\\[3em]
{\normalsize \bf Emil Wiedemann}\\[0.5em]
\small Department of Mathematics, Friedrich-Alexander-Universit\"at Erlangen-N\"urnberg, Germany.\\
E-mail: emil.wiedemann@fau.de\\[3em]

{\em Dedicated to the memory of Olga Aleksandrovna Ladyzhenskaya on the occasion of her 100th birthday}\\[3em]
\end{flushleft}
{\bf Abstract:} Fluids can behave in a highly irregular, turbulent way. It has long been realised that, therefore, some weak notion of solution is required when studying the fundamental partial differential equations of fluid dynamics, such as the compressible or incompressible Navier-Stokes or Euler equations. The standard concept of weak solution (in the sense of distributions) is still a deterministic one, as it gives exact values for the state variables (like velocity or density) for almost every point in time and space. However, observations and mathematical theory alike suggest that this deterministic viewpoint has certain limitations. Thus, there has been an increased recent interest in the mathematical fluids community in probabilistic concepts of solution. Due to the considerable number of such concepts, it has become challenging to navigate the corresponding literature, both classical and recent. We aim here to give a reasonably concise yet fairly detailed overview of probabilistic formulations of fluid equations, which can roughly be split into {\em measure-valued} and {\em statistical} frameworks. We discuss both approaches and their relationship, as well as the interrelations between various statistical formulations, focusing on the compressible and incompressible Euler equations.   


\tableofcontents

\vspace{1cm}

\section{Introduction}

When the flow of a fluid, whether compressible or incompressible, viscous or ideal, is described within the setting of classical continuum mechanics, one would generally expect to obtain a deterministic theory: If the state of the fluid is fully known at a certain time, then it should in principle be possible to uniquely calculate its state at any subsequent time. This seems in line with the description of fluid flows by means of nonlinear partial differential equations, such as the {\em isentropic compressible Euler equations}
\begin{equation}\label{eq: CE}
\begin{aligned}
\partial_t(\rho u)+\operatorname{div}(\rho u\otimes u)+\nabla p(\rho)&=0,\\
\partial_t\rho + \operatorname{div}(\rho u)&=0,
\end{aligned}
\end{equation}    
or the {\em incompressible Euler equations}
\begin{equation}\label{eq: IE}
\begin{aligned}
\partial_tu+\operatorname{div}(u\otimes u)+\nabla p&=0\\
\operatorname{div} u&=0.
\end{aligned}
\end{equation} 
These equations can be posed either on an open subset $\Omega\subset\R^d$ or on the torus $\T^d$, where the space dimension is $d\geq 2$ in the incompressible and $d\geq 1$ in the compressible case; either way, for some given time $T>0$, one looks for the unknown velocity field $u:(0,T)\times\Omega\to\R^d$ and, for~\eqref{eq: CE}, also for the unknown density field $\rho:(0,T)\times\Omega\to\R$, which is always non-negative (non-negativity is propagated in time, so it suffices to impose a non-negative initial density). The role of the pressure is different in the two models: In~\eqref{eq: CE} it is a constitutively given function of the density, while in~\eqref{eq: IE} it is another unknown $p:(0,T)\times\Omega\to\R$. If the boundary of $\Omega$ is non-empty, one usually imposes the impermeability boundary condition $u\cdot n=0$, where $n:\partial\Omega\to\R^d$ is the outer unit normal. This means the fluid may not flow in or out of the given domain.

Both systems are formally well-posed, as the number of unknowns matches the number of equations. The same is true for the corresponding viscous models, where~\eqref{eq: CE} and~\eqref{eq: IE} are supplemented with suitable friction terms, respectively. In the case of Newtonian fluids, one thus arrives at the (compressible or incompressible) {\em Navier-Stokes equations}. 

We will deal here primarily with the inviscid models, although much of our discussion will not be specific for the Euler equations. Our choice has several reasons: First, turbulent flows typically occur at high Reynolds numbers, and the Euler equations correspond to the limiting case of infinite Reynolds number. Secondly, from a mathematical viewpoint, the Euler equations display vast ill-posedness in the framework of weak solutions. Both these observations, arguably, shake our na\"ive attempt to describe turbulent flows deterministically.

Indeed, consider an experiment where water flows past an obstacle. In the wake of the obstacle, a turbulent region is observable. Repeating the experiment many times, one will always see a different flow field in the turbulent region; it is not deterministically predictable, at least not practically. Yet, qualitatively, the result of the experiment appears similar each time. It turns out that, while the details of the flow field are unpredictable, certain {\em statistical} quantities of the flow appear to be very regular and stable, such as the average velocity or the so-called structure functions. For a discussion of these observations, see~\cite{frisch}; extensive state-of-the-art numerical simulations supporting similar conclusions can be found in~\cite{FMT}.

This line of thought is very classical and -- at least in the incompressible situation -- has led to statistical, phenomenological theories of turbulence, most prominently the one of Kolmogorov~\cite{K42}. The relation between Kolmogorov's theory and the arguably more fundamental PDE-based approach is still very unclear, although the recent rigorous proof of Onsager's Conjecture~\cite{eyink, CET, Isett, BDSV} provides a fascinating link.

But also from the analytical viewpoint, there have arisen issues with the deterministic weak solution concept. It has been known since the work of Scheffer~\cite{scheffer} that weak solutions of the incompressible Euler equations can evolve non-uniquely from the same initial data. While Scheffer's solutions can be discarded as non-physical artefacts, convex integration techniques developed by De Lellis and Sz\'ekelyhidi have allowed to show that the Cauchy problem remains ill-posed even under further energy conditions~\cite{DLS08}, and that this is the case for a dense subset (in the energy space) of initial data~\cite{SW12}, even in regularity classes up to the Onsager-critical H\"older exponent~\cite{DRSz}.
The situation is similar for the compressible Euler equations~\cite{DLS08, Chio, CDLK, CKMS, CVY}. 

A conceivable way out of all this trouble is to turn to {\em probabilistic} descriptions of fluid flows, thus abandoning the noble ambition of predicting exactly the fluid velocity (and density) field at every time. Instead, for a probabilistic framework, at least three different options spring to mind:

\begin{itemize}
\item {\bf trajectory statistical solutions:} a probability measure $\mu$ on the set of solution trajectories (say of the velocity) of the PDE, for instance on the space $C([0,T];H)$, where $H$ is the space of solenoidal vector fields in $L^2(\mathbb R^3;\mathbb R^3)$;
\item {\bf phase space statistical solutions:} a time-dependent probability measure $\mu_t$ which, for every time $t$, lives on phase space (say on $H$) and whose evolution is governed in some way by the underlying PDE;
\item {\bf measure-valued solutions (mvs):} a probability measure $\mu_{t,x}$, depending on time and space, that lives on $\mathbb R^3$ and whose dependence on $x$ and $t$ is again constrained by the underlying equation.  
\end{itemize}  

It turns out that the first two concepts are closely related, while the measure-valued solutions have a very different character. This is because the latter describe only the \emph{one-point statistics} of the flow: One can express, in the language of measure-valued solutions, a statement like ``The probability that, at time $t$ and location $x$, the fluid velocity has negative vertical component equals 30\%'', but not a statement about {\em correlations} of the velocity between two different points in space, such as ``The probability that the vertical velocity components at time $t$ at space points $x$ and $y$ have the same sign is 30\%''. Such information on correlations can be extracted from statistical solutions. This indicates that statistical solutions contain more information and thus represent a stronger notion than the measure-valued concept. Mathematically, this comes at a price: Measure-valued solutions, in contrast to statistical ones, can easily be shown to exist for any initial data of finite energy. In fact, known results indicate that the existence theory for statistical solutions is essentially just as hard as for deterministic weak solutions. 

Before we describe in more detail the properties of measure-valued and of statistical solutions to be discussed in the body of this work, let us point out what will {\em not} be covered here: First, despite the probabilistic flavour of the theory, no actual randomness -- and, correspondingly, no genuinely stochastic mathematical techniques -- will appear. Indeed, although the initial data will typically be a probability measure instead of a deterministic function, the evolution itself will not be subject to any stochastic forcing, for instance. Secondly, in the context of statistical solutions of the Navier-Stokes equations, one could be interested in the long-time behaviour of the system from a dynamical systems perspective and talk about statistically stationary solutions, invariant manifolds, global attractors, etc. (see for instance~\cite[Chapters III \& IV]{FMRT01}). We shall not say anything about either of these topics, as we focus on fundamental existence and uniqueness questions and on the interrelations between various probabilistic formulations of the Euler equations. 

\subsection{Measure-Valued Solutions of the Euler Equations}
Generally speaking, measure-valued solutions are parametrized probability measures that solve a partial differential equation only in an averaged sense. In particular, a (weak) solution can be recovered as a measure-valued solution if the latter is equal to a Dirac distribution. Hence, measure-valued solutions at a certain point in space-time consist of a probability distribution of possible values instead of one specific value. As a consequence, we consider measure-valued solutions as an instance of a probabilistic type of solution. Note that the evolution of the averages of measure-valued solutions is constrained by the PDEs. The rigorous definition of this very weak notion of solution will be given below in Section~\ref{defmvs}.

Measure-valued solutions build upon the concept of Young measures, which are certain parametrized probability measures. They are named after L.~C.~Young, who introduced them in the context of optimal control theory~\cite{Y37}. Young measures are also an effective tool to study the limiting behavior of oscillating sequences in the calculus of variations and PDE theory, cf.~for example \cite{M99,T79}. As a solution concept, measure-valued solutions consisting of Young measures were introduced by DiPerna \cite{D85} for hyperbolic conservation laws, where they were used to give a proof for the existence of energy admissible weak solution in space dimension $d=1$ by using compensated compactness.

Measure-valued solutions of the incompressible Euler system have been introduced and proven to exist for any initial data by DiPerna and Majda \cite{DM87osc}, where the main technical advance was to include possible concentrations. This is not quite the framework we will discuss here, but we will focus on the formalism of~\cite{FGSW17}, cf.~also \cite{DST12}. Neustupa \cite{N93} then defined measure-valued solutions for the compressible Euler and Navier-Stokes equations and also proved their existence. Note that the existence of measure-valued solutions is particularly interesting whenever weak solutions are not known to exist for all initial data, which is the case for the incompressible and compressible Euler equations and the 3D compressible Navier-Stokes system with adiabatic exponent $\gamma\leq 3/2$.

Arguably, the notion of measure-valued solution is a very weak solution concept. This is reflected by its vast non-uniqueness: The Young measure is only constrained in its first and second moments (see Definition~\ref{defmvsdef}), whereas for given expectation and (non-zero) covariance, there is always an infinitude of probability measures. On top of this obvious non-uniqueness, a further degree of ill-posedness is inherited by the non-uniqueness of weak solutions created through convex integration, cf.~for example \cite{DLS08}. However, somewhat surprisingly, measure-valued solutions still enjoy the weak-strong uniqueness property, see \cite{BDS11} for a proof in the incompressible and \cite{GSW15} for a proof in the compressible Euler case. A more practical use of measure-valued solutions is the description of singular limits of approximating solution sequences. For example, certain numerical schemes for the compressible Navier-Stokes system converge to measure-valued solutions as shown in \cite{FL18}. See also \cite{FKMT17} for a discussion of this topic regarding the vortex-sheet initial data. Other measure-valued singular limits are also studied in \cite{BTW12,FKM19,G22}.

Moreover, for the incompressible Euler equations in space dimension $d\geq 2$ it holds that every oscillation measure-valued solution is the limit of some sequence of weak solutions, cf.~Theorem~\ref{thm: incompressiblemvs} below. Thus, although measure-valued solutions are only constrained through their first and second moment, they contain essentially the same information as weak solutions.

The latter property has also been a point of criticism in the sense that measure-valued solutions only display one-point statistics. However, for an adequate description of turbulence one may also include correlations. 




\subsection{Statistical Solutions of the Euler Equations}
Statistical solutions are generally speaking a probabilistic concept which aims at describing the evolution of a whole ensemble of initial data or a whole ensemble of solutions governed by some PDE.

In the former case, for the 3D Navier-Stokes equations, a first attempt at this was made by E. Hopf in \cite{H52}: Given a probability distribution $\overline{\mu}$ on a set of initial data and assuming that all individual phases $u$ have been transformed to $S_tu$ after time $t \geq 0$, the distribution $\mu_t$ of the solutions at time $t$ should be determined by the formula
\begin{equation}\label{eq: pushforward solutions}
\mu_t(A) = \overline{\mu}(S_t^{-1}A)
\end{equation}
for all measurable sets $A$.

As probability distributions are uniquely determined by their characteristic functional, i.e. their Fourier transform 
\begin{equation*}
\F(v,t) = \int \exp(i(v,u))\dd\mu_t(u) = \int \exp(i(v,S_tu))\dd\overline{\mu}(u),
\end{equation*}
Hopf went on and formally derived an equation for $\F(v,t)$, the \textit{Hopf statistical equation}, which would no longer explicitly depend on the solution operators $S_t$, $t\geq 0$. In the natural phase space $L^2$ of the Navier-Stokes equations, one can express elements $v$ of this Hilbert space in terms of abstract Fourier series with respect to suitable orthonormal bases and consequently transform the Hopf statistical equation on an infinite dimensional space into an equation for the infinitely many (finite-dimensional) Fourier coefficients of the argument $v$ of $\F$. This was considered and rigorously solved in 1976 by Ladyzhenskaya and Vershik~\cite{LV78}. We will derive this \textit{coordinate version} later for the Euler equations.

A few years prior, in the two seminal articles \cite{F72,F73}, Foia\cs\, had proposed the first rigorous concept of statistical solutions of the Navier-Stokes equations, based on discussions with Prodi. Just like Hopf, Foia\cs\, described time-parametrized probability measures $\lbrace \mu_t \rbrace_{t \geq 0}$ on an $L^2$ based phase space for a given initial distribution $\overline{\mu}$. Likewise, in order to formally derive an equation describing the evolution of these measures, the assumption of existence of solution operators $\lbrace S_t\rbrace_{t \geq 0}$ describing the deterministic evolution was made so that naturally $\mu_t$ is the pushforward measure of $\overline{\mu}$ along $S_t$ as in \eqref{eq: pushforward solutions}, which we will denote by ${S_t}_\sharp\overline{\mu}$. However, Foia\cs\, did not consider their Fourier transforms but more generally $\int \Phi(u)\dd\mu_t(u) = \int \Phi(S_tu)\dd\overline{\mu}(u)$ for a suitable class of functionals $\Phi$. Then, the evolution equation $\partial_t u = F(t,u)$ dictates 
\begin{align*}
\frac{d}{dt}\int \Phi(u)\dd\mu_t(u) &= \frac{d}{dt}\int \Phi(S_tu)\dd\overline{\mu}(u) = \int (\Phi'(S_tu),\partial_t(S_tu))_{L^2}\dd\overline{\mu}(u)\\
&= \int (\Phi'(S_tu),F(t,S_tu))_{L^2}\dd\overline{\mu}(u) = \int (\Phi'(u),F(t,u))_{L^2}\dd\mu_t(u).
\end{align*}
Integrating the left and the right-hand sides over some time interval $[t',t]$, $0 \leq t' \leq t$, yields the \textit{Foia\cs-Liouville equation}
\begin{equation}\label{eq: Foias-Liouville}
\int \Phi(u)\dd\mu_t(u) = \int \Phi(u)\dd\mu_{t'}(u) + \int_{t'}^t \int (\Phi'(u),F(t,u))_{L^2}\dd\mu_s(u)\,\dd s.
\end{equation}
As suggested above, the Foia\cs-Prodi approach seems to be more general and flexible as Hopf basically only considered the functionals $\Phi(u) = \exp(i(v,u))$ with fixed $v \in L^2(\R^2)$. These functionals can be approximated by Foia\cs' admissible test functionals, as noted by himself in \cite{F72}, and therefore Foia\cs-Prodi statistical solutions also solve the Hopf statistical equation.

A few years thereafter, Vishik and Fursikov introduced a different type of statistical solution of the Navier-Stokes equations, see \cite{VF79,VF88}. There, they consider single distributions on the set of solution trajectories of the Navier-Stokes equations and show by a similar approach their existence.

A connection to the Foia\cs-Prodi statistical solutions has been drawn by Foia\cs, Rosa and Temam in \cite{FRT13}, where they introduce a very similar type of statistical solution, labelled Vishik-Fursikov measure, but with some differences to the original notion of Vishik and Fursikov. For example Foia\cs, Rosa and Temam suppose a sharper type of energy inequality and show that the set of Leray-Hopf solutions of the Navier-Stokes equations is already measurable, instead of assuming that the Leray-Hopf solutions contain a measurable subset which carries the measure.

Projecting these Vishik-Fursikov measures to phase space at each time yields a family of time-parametrized measures, which are then indeed Foia\cs-Prodi statistical solutions. It is unclear if every Foia\cs-Prodi statistical solution can be obtained in this way.

We would also like to point out that the method of constructing Vishik-Fursikov measures in \cite{FRT13}, which is based on the Krein-Milman theorem, is flexible enough to also be applied to other equations. Indeed, it has served as motivation for Bronzi, Mondaini and Rosa to develop an abstract framework for the theory of statistical solutions \cite{BMR16}, which has also helped the last two authors of this article construct statistical solutions of the 2D incompressible Euler equations in a fairly comprehensive manner~\cite{WW23}. In this broader context, the analogue of the Vishik-Fursikov measure is called {\em trajectory statistical solution} and the analogue of the Foia\cs-Prodi statistical solution is labelled {\em phase space statistical solution}.

At this point, the connection to measure-valued solutions is not really clear. A first such connection was drawn by Chae in \cite{C91g}, where he proved that in the three-dimensional case, Foia\cs-Prodi statistical solutions of the Navier-Stokes equations converge in the vanishing viscosity limit to a measure-valued solution of the Euler equations, as introduced in \cite{DM87conc}. In particular, Chae showed that certain time-parametrized measures on $L^2$ correspond to (generalized) Young measures.

It is still unclear whether the converse is possible, namely, if every (generalized) Young measure corresponds to a time parametrized measure on $L^2$. Related to this question is the work \cite{FLM17} by Fjordholm, Lanthaler and Mishra. There, they introduce a new notion of statistical solution of hyperbolic conservation laws, based on infinite hierarchies of Young measures, which they label \textit{correlation measures}. In particular, they prove that there is a one-to-one correspondence between correlation measures and measures on $L^p$ spaces.

As for the connection to measure-valued solutions, the correlation measures are not specifically tied to hyperbolic conservation laws and in fact, a formulation similar to their original definition may be introduced for the Euler or Navier-Stokes equations~\cite{FW18,LMP21,FMW22}. This formulation leads to an infinite chain of equations, where the first equation precisely demands that the first Young measure in this hierarchy be a measure-valued solution in the sense of DiPerna~\cite{D85} (without concentration part).

This has to be interpreted in the way that the correlation measure augments the classical Young measure or measure-valued solutions -- which only describe the first moments and one-point statistics -- with all multipoint correlations and therefore carries more information.

For the sake of a priori distinguishing this concept of statistical solutions from the other ones, we will refer to them as \textit{moment based statistical solutions}. In \cite{FMW22}, in the context of the Euler and Navier-Stokes equations, these have been labelled \textit{(inviscid) Friedmann-Keller statistical solutions}. This is due to early work by Keller and Friedmann in 1924~\cite{FK24} and later considerations by Fursikov, who also used this term for describing moment based chains of evolution equations related to the Navier-Stokes system, see e.g.~\cite{F87,F90,F91,F95}.

We mentioned before that there is a one-to-one correspondence between measures on $L^p$ spaces and correlation measures. As the latter objects, parametrized by time, are considered for the notion of statistical solutions due to Foia\cs\, and Prodi or more generally, the aforementioned phase space statistical solutions, this one-to-one correspondence allows one to compare these two notions of statistical solutions. It turns out in fact that in case of the Euler or Navier-Stokes equations, the two notions of statistical solutions are equivalent. In the latter case, this can be found in the recent preprint \cite{FMW22} due to Fjordholm, Mishra and Weber. In case of the Euler equations, this will be proved here later, as it illustrates the main arguments but is a bit easier, since one does not have to worry about ensuring that the correlation measure or the associated measure on $L^2$ is concentrated on $H^1$.

In Subsection~\ref{subsec: definitions}, we will give the precise definition of each notion of statistical solution that we mentioned in this introduction and compare all of them in more detail. In the subsequent subsection, we discuss and compare various ways in which statistical solutions can be constructed. Finally, in Subsection~\ref{Comparison to measure valued solutions}, we will comment more in depth on the connection between measure-valued and statistical solutions, and in Section~\ref{Discussion and open problems}, we develop some more thoughts and open problems in this direction.

\section{Measure-Valued Solutions}\label{sec: mvs}

\subsection{Definitions for the Incompressible and the Isentropic Euler System}\label{defmvs}
Let us quickly review the basics from Young measure theory needed for our discussion.\\
Let $\Omega\subset\R^d$ be open (or $\Omega=\T^d$) and let $X\subset \R^N$ be a measurable set. A \textit{Young measure} $\nu$ on $X$ with parameters in $(0,T)\times \Omega$ is defined to be an element of $L^\infty_{\operatorname{w}}((0,T)\times \Omega;\mathcal{P}(X))$. Here, $L^\infty_{\operatorname{w}}((0,T)\times \Omega;\mathcal{P}(X))$ denotes the set of weakly-* measurable families of probability measures $\nu=(\nu_{t,x})_{(t,x)\in(0,T)\times \Omega}$, i.e.~the map
\begin{align*}
	(t,x)\mapsto \int\limits_{X}^{}f(z)\dd\nu_{t,x}(z)=:\langle\nu_{t,x},f\rangle
\end{align*}
is measurable for all functions $f\in C_0(X)$.\\
There is also a notion of convergence tailored to Young measures:\\
We say that a sequence of measurable functions $(z_n)_{n\in\N}$ from $(0,T)\times \Omega$ to $X$ \textit{generates} the Young measure if
\begin{align*}\label{equ: generating YM}
	f(z_n(\cdot))\overset{*}{\rightharpoonup}\langle\nu_{(\cdot)},f\rangle\text{ in }L^{\infty}((0,T)\times \Omega)
\end{align*}
for all $f\in C_0(X)$. As a shorthand notation we will use $z_n\overset{Y}{\rightharpoonup}\nu$.\\
This notion of convergence, by definition, behaves very well under composition with nonlinear functions, unlike weak limits. This may be the most important advantage of using solution concepts based on Young measures for nonlinear PDEs. For example, consider the stationary vortex sheet
\begin{align*}
	u(t,x_1,x_2)=\begin{cases}
		e_1,& x_2<\frac{1}{2}\\
		-e_1,& x_2>\frac{1}{2}
	\end{cases},
\end{align*}
which solves the incompressible Euler equations on the torus $\T^2$. As a consequence, for every $n\in \N$, the functions $u_n(t,x_1,x_2):=u(nt,nx_1,nx_2)$ are also solutions. Moreover, $(u_n)_{n\in\N}$ converges weakly to zero. However, if we compose it with the nonlinear operation $f\colon v\mapsto v\otimes v$, the constant functions $u_n\otimes u_n=\begin{pmatrix}
	1 & 0\\
	0 & 0
\end{pmatrix}$ clearly have a nonzero limit.\\
On the other hand, it is not hard to show that
\begin{align*}
	u_n\overset{Y}{\rightharpoonup}\frac{1}{2}\delta_{e_1}+\frac{1}{2}\delta_{-e_1}.
\end{align*}
Since $(u_n)_{n\in\N}$ is uniformly bounded in $L^{\infty}$, we can still insert the tensor product $f\colon v\mapsto v\otimes v$ into the definition of Young measure convergence, which yields
\begin{align*}
	u_n\otimes u_n=f(u_n)\overset{*}{\rightharpoonup}\langle\nu, f\rangle = \begin{pmatrix}
		1 & 0\\
		0 & 0
	\end{pmatrix}.
\end{align*}
The following convergence result is essential and is a consequence of the Banach-Alaoglu Theorem.
\begin{theorem}[Fundamental theorem of Young measure theory]\label{thm: FundamentalYM}
	Let $(u_n)_{n\in\N}$ be a sequence of maps bounded in $L^p((0,T)\times \Omega;\R^N)$ for some $1\leq p\leq\infty$. Then, up to a subsequence, $(u_n)_{n\in\N}$ generates a Young measure on $\R^N$.
\end{theorem}
Note that every element of $L_{\operatorname{w}}^\infty((0,T)\times \Omega;\mathcal{P}(X))$ can be shown to be generated by some sequence of measurable functions.

With Theorem~\ref{thm: FundamentalYM} at hand we are able to introduce the notion of measure-valued solution. We begin with defining measure-valued solutions for the incompressible Euler system.
\begin{definition}\label{defmvsdef}
	A triple $(\nu,m,D)$ is a \emph{dissipative measure-valued solution} of the incompressible Euler system (\ref{eq: IE}) with initial data $u^0\in L^2(\Omega)$ if
	\begin{itemize}
		\item[i)] $\nu\in L^{\infty}_{\operatorname{w}}((0,T)\times \Omega;\mathcal{P}(\R^d))$ is a Young measure,
		\item[ii)] the measure $m\in \mathcal{M}((0,T)\times \overline{\Omega};\R^{d\times d})$ satisfies
		\begin{align*}
			m(\dx\dt\,)=m_t(\dx\,)\otimes \dt
		\end{align*}
		for some family of measures $(t\mapsto m_t)\in L^{\infty}\left((0,T);\mathcal{M}(\overline{\Omega};\R^{d\times d})\right)$,
		\item[iii)] $D\in L^{\infty}(0,T)$ satisfies $D\geq 0$ and
		\begin{align*}
			|m_t|(\overline{\Omega})\leq C\cdot D(t)
		\end{align*}
		for some $C>0$ and a.e.~$t\in(0,T)$,
		\item[iv)] the momentum equation
		\begin{align*}
			\int\limits_{0}^{T}\int\limits_{\Omega}^{}\langle\nu_{t,x},\operatorname{id}\rangle\cdot\partial_t\varphi&+\langle\nu_{t,x},\operatorname{id}\otimes\operatorname{id}\rangle:\nabla\varphi\dx\dt+\int\limits_{0}^{T}\int\limits_{\Omega}^{}\nabla\varphi:m_t\dt\\
			&=-\int\limits_{\Omega}^{}u^0\cdot\varphi(0,\cdot)\dx
		\end{align*}
		holds for all $\varphi\in C_c^{\infty}([0,T)\times \Omega;\R^d)$ with $\diverg \varphi=0$,
		\item[v)] the divergence-free condition
		\begin{align*}
			\int\limits_{\Omega}^{}\langle\nu_{t,x},\operatorname{id}\rangle\cdot\nabla\psi\dx=0
		\end{align*}
		holds for a.e.~$t\in(0,T)$ and all $\psi\in C_c^{\infty}(\Omega)$,
		\item[vi)] the energy inequality
		\begin{align*}
			\frac{1}{2}\int\limits_{\Omega}^{}\langle\nu_{t,x},|\cdot|^2\rangle\dx+D(t)\leq \frac{1}{2}\int\limits_{\Omega}^{}|u^0|^2\dx
		\end{align*}
		holds for a.e.~$t\in(0,T)$.
	\end{itemize}
	As usual, the above integrals are assumed to exist as part of the definition. When $m=0$ and $D=0$, we simply write $\nu$ and call it an \emph{oscillation measure-valued solution}. Moreover, if an oscillation measure-valued solution $\nu$ satisfies $\nu_{t,x}=\delta_{u(t,x)}$ for some $u\in L^2((0,T)\times \Omega;\R^d)$, this boils down to the usual definition of $u$ being a weak solution.
\end{definition}
The corresponding notion of measure-valued solution for the isentropic Euler system is defined as follows.
\begin{definition}\label{def: mvsCE}
	A triple $(\nu,m,D)$ is a \emph{dissipative measure-valued solution} of the isentropic Euler system (\ref{eq: CE}) with adiabatic exponent $1<\gamma<\infty$ and initial data $(\rho^0,u^0)\in L^1(\Omega)$ if
	\begin{itemize}
		\item[i)] $\nu\in L^{\infty}_{\operatorname{w}}((0,T)\times \Omega;\mathcal{P}([0,\infty)\times \R^d))$ is a Young measure,
		\item[ii)] the measure $m\in \mathcal{M}((0,T)\times \overline{\Omega};\R^{d\times d})$ satisfies
		\begin{align*}
			m(\dx\dt)=m_t(\dx)\otimes \dt
		\end{align*}
		for some family of measures $(t\mapsto m_t)\in L^{\infty}\left((0,T);\mathcal{M}(\overline{\Omega};\R^{d\times d})\right)$,
		\item[iii)] $D\in L^{\infty}(0,T)$ satisfies $D\geq 0$ and
		\begin{align*}
			|m_t|(\overline{\Omega})\leq C\cdot D(t)
		\end{align*}
		for some $C>0$ and a.e.~$t\in(0,T)$,
		\item[iv)] the momentum equation
		\begin{align*}
			&\int\limits_{0}^{T}\int\limits_{\Omega}^{}\langle\nu_{t,x},\rho u\rangle\cdot\partial_t\varphi+\langle\nu_{t,x},\rho u\otimes u\rangle:\nabla\varphi+\langle\nu_{t,x},\rho^{\gamma}\rangle\diverg\varphi\dx\dt\\
			&+\int\limits_{0}^{T}\int\limits_{\Omega}^{}\nabla\varphi:\dd m_t\dt
			=-\int\limits_{\Omega}^{}\rho^0u^0\cdot\varphi(0,\cdot)\dx
		\end{align*}
		holds for all $\varphi\in C_c^{\infty}([0,T)\times \Omega;\R^d)$,
		\item[v)] the continuity equation
		\begin{align*}
			\int\limits_{0}^{T}\int\limits_{\Omega}^{}\langle\nu_{t,x},\rho\rangle\cdot\partial_t\psi+\langle\nu_{t,x},\rho u\rangle\cdot\nabla\psi\dx\dt=-\int\limits_{\Omega}^{}\rho^0\psi(0,x)\dx
		\end{align*}
		holds for all $\psi\in C_c^{\infty}([0,T)\times\Omega)$,
		\item[vi)] the energy inequality
		\begin{align*}
			\frac{1}{2}\int\limits_{\Omega}^{}\langle\nu_{t,x},\rho|u|^2\rangle+\frac{1}{\gamma-1}\langle\nu_{t,x},\rho^{\gamma}\rangle\dx+D(t)
			\leq \frac{1}{2}\int\limits_{\Omega}^{}\rho^0|u^0|^2+\frac{1}{\gamma-1}(\rho^0)^{\gamma}\dx
		\end{align*}
		holds for a.e.~$t\in(0,T)$.
	\end{itemize}
	Again, the above integrals are assumed to exist as part of the definition. When $m=0$ and $D=0$, we simply write $\nu$ and call it an \emph{oscillation measure-valued solution}.	Moreover, if an oscillation measure-valued solution $\nu$ satisfies $\nu_{t,x}=\delta_{(\rho(t,x),u(t,x))}$ for some $(\rho,u)\in L^1((0,T)\times \Omega;[0,\infty)\times \R^d)$, this reduces to the usual definition of $(\rho,u)$ being a weak solution.
\end{definition}

\subsection{Existence of Measure-Valued Solutions}
One of the most striking advantages of the notion of measure-valued solution is that for every $L^2$-initial data a corresponding dissipative measure-valued solution is known to exist globally. This is still unknown for weak solutions. If one drops the requirement of energy dissipativity, then at least for the incompressible Euler system weak solutions exist for every initial data, cf.~\cite{W11}. 
\begin{theorem}
	Let $u^0\in L^2(\Omega)$ or $(\rho^0,u^0)\in (L^{\gamma}\times L^2)(\Omega)$. Then there exists a dissipative measure-valued solution of the incompressible Euler system with initial data $u^0$, and in the compressible case there exists a dissipative measure-valued solution of the isentropic Euler system with initial data $(\rho^0,u^0)$.
\end{theorem}
\begin{proof}
	We will only prove the case of the isentropic Euler system when energy admissible weak solutions of the compressible Navier-Stokes system are known to exist, e.g.~when $\gamma>\frac{3}{2}$ and $d=3$ by the existence theory from \cite{FNP01}. The incompressible case follows similarly, where weak Navier-Stokes solutions exist by the work of Leray-Hopf \cite{L34,H51}. The other cases for the isentropic Euler equations are covered by \cite{FGSW17}.
	
	So, let $(\rho^0,u^0)\in (L^{\gamma}\times L^2)(\Omega)$ and let $(\rho_n,u_n)$ be the corresponding energy admissible weak solutions of the compressible Navier-Stokes system with viscosity parameter $\alpha_n\rightarrow 0$. This means that $(\rho_n,u_n)$ satisfy
	\begin{align*}
		\partial_t (\rho_n u_n)+\diverg (\rho_nu_n\otimes u_n)+\nabla(\rho_n^{\gamma})&=\alpha_n\diverg \mathbb{S}(\nabla u_n)\\
		\partial_t \rho_n+\diverg (\rho_n u_n)&=0
	\end{align*}
	in the sense of distributions with energy bound
	\begin{align*}
		&\;\int\limits_{\Omega}^{}\frac{1}{2}\rho_n|u_n|^2+\frac{1}{\gamma-1}(\rho_n)^{\gamma}\dx+\int\limits_{0}^{t}\int\limits_{\Omega}^{}\alpha_n\mathbb{S}(\nabla u_n):\nabla u_n\dx\dt\\
		\leq&\;\int\limits_{\Omega}^{}\frac{1}{2}\rho^0|u^0|^2+\frac{1}{\gamma-1}(\rho^0)^\gamma\dx=:E_0
	\end{align*}
	for a.e.~$t\in (0,T)$, where
	\begin{align*}
		\mathbb{S}(A)=\kappa\left(A+A^{\operatorname{T}}+\frac{2}{3}(\operatorname{tr}A)\cdot\mathbb{E}_d \right)+\eta(\operatorname{tr}A)\mathbb{E}_d.
	\end{align*}
	Here, $\mathbb{E}_d$ denotes the identity matrix in $d$ dimensions and $\kappa,\eta$ are some fixed positive constants. (For purely mathematical purposes, setting $\mathbb{S}(A)=A$ would do just as well.)
	
	The uniform energy bound implies that $(\rho_n)_{n\in\N}$ is $L^{\infty}_tL^{\gamma}_x$-bounded and $(\sqrt{\rho_n}u_n^2)_{n\in\N}$ is $L^{\infty}_tL^2_x$-bounded. Thus, Theorem~\ref{thm: FundamentalYM} implies that there exists a subsequence (which we will not relabel) that generates a Young measure $\nu$.
	
	When we consider $(\rho_nu_n\otimes u_n)$ and $(\rho_n^{\gamma})\mathbb{E}_d$, we observe that these sequences are bounded in $L^{\infty}((0,T);L^1(\Omega))$ but may not be bounded. Therefore, the functions $(\rho,\sqrt{\rho}u)\mapsto \rho u\otimes u$ and $\rho\mapsto \rho^\gamma$ may not be admissible test functions in the definition of Young measure generation, and accordingly the terms
	\begin{align*}
		\rho_nu_n\otimes u_n-\langle\nu_{t,x},\rho u\otimes u\rangle\text{ and }(\rho_n)^{\gamma}\mathbb{E}_d-\langle\nu_{t,x},\rho^{\gamma}\mathbb{E}_d\rangle
	\end{align*}
	may not converge weakly-* to zero. On the other hand, we may interpret the above terms as bounded Radon-measures on $[0,T]\times \overline{\Omega}$. Then we obtain (not relabeled) subsequences converging weakly-* in the sense of Radon-measures to some $m^{\rho u\otimes u},m^{\rho^{\gamma}}\in \mathcal{M}([0,T]\times \overline{\Omega};\R^{d\times d})$, respectively. Standard arguments from measure theory imply from the $L^{\infty}_tL^1_x$-bound a \textit{disintegration} of the form
	\begin{align*}
		m^{\rho u\otimes u}(\dx\dt\,)&=m_t^{\rho u\otimes u}(\dx\,)\otimes \dt,\\
		m^{\rho^{\gamma}}(\dx\dt\,)&=m_t^{\rho^{\gamma}}(\dx\,)\otimes \dt,
	\end{align*}
	where $(m_t^{\rho u\otimes u}),(m_t^{\rho^{\gamma}})$ are uniformly bounded families of measures in $\mathcal{M}(\overline{\Omega};\R^{d\times d})$. We define $m:=m^{\rho u\otimes u}+m^{\rho^{\gamma}}$. Then the disintegration measures also add up, i.e.~$m_t=m_t^{\rho u\otimes u}+m_t^{\rho^{\gamma}}$.
	
	We also define the \textit{dissipation defect} $D$ in terms of the trace of $m$ as the uniformly bounded function
	\begin{align*}
		D(t):=\frac{1}{2}\operatorname{tr}\left(m_t^{\rho u\otimes u}(\overline{\Omega})\right)+\frac{1}{d(\gamma-1)}\operatorname{tr}\left(m_t^{\rho^{\gamma}}(\overline{\Omega})\right)
	\end{align*}
	for a.e. $t\in (0,T)$. Since both $m_t^{\rho u\otimes u}$ and $m_t^{\rho^{\gamma}}$ live on the convex set of positive semidefinite symmetric matrices, any matrix norm is equivalent to the trace and so there exists some $C>0$ such that
	\begin{align*}
		|m_t|(\overline{\Omega})\leq C\cdot D(t)
	\end{align*}
	for a.e. $t\in(0,T)$. In particular, $D\geq 0$.
	We will now show that $(\nu,m,D)$ is a dissipative measure-valued solution.
		
	Note that i), ii) and iii) in Definition~\ref{def: mvsCE} are satisfied by construction. Since $(\rho_n)_{n\in\N}$ is $L^{\infty}_tL^{\gamma}_x$-bounded, it is equiintegrable, as $\Omega$ is a bounded set. By H\"older's inequality and the kinetic energy bound $\sqrt{\rho_n}u_n\in L^\infty L^2$, we see that $(\rho_nu_n)_{n\in\N}$ is bounded in $L^{2\gamma/(\gamma+1)}$, hence it is equiintegrable as $\frac{2\gamma}{\gamma+1}>1$ for $\gamma>1$.
	Therefore,
	\begin{align*}
		\int\limits_{0}^{T}\int\limits_{\Omega}^{}\rho_n\partial_t\psi+\rho_nu_n\cdot\nabla\psi\dx\dt\rightarrow \int\limits_{0}^{T}\int\limits_{\Omega}^{}\langle\nu_{t,x},\rho\rangle\partial_t\psi+\langle\nu_{t,x},\rho u\rangle\cdot\partial_t\psi\dx\dt
	\end{align*}
	for any $\psi\in C_c^{\infty}((0,T)\times \Omega;\R)$. This shows v) in Definition~\ref{def: mvsCE}.\\
	For the momentum equation, observe that for fixed $\varphi\in C_c^{\infty}([0,T)\times \Omega;\R^d)$ we have
	\begin{align*}
		&\;\alpha_n\int\limits_{0}^{T}\int\limits_{\Omega}^{}\mathbb{S}(\nabla u_n):\nabla\varphi\dx\dt\leq C\alpha_n\|\nabla\varphi\|_{L^{\infty}_tL^2_x}\int_0^T\|\nabla u_n(t)\|_{L^2_x}\dt\\
		\leq &\;C\alpha_n\|\nabla\varphi\|_{L^{\infty}_tL^2_x}\int_0^T\left(\int\limits_{\Omega}^{}\mathbb{S}(\nabla u_n):\nabla u_n\dx \right)^{\frac{1}{2}}\dt\|\nabla\varphi\|_{L^{\infty}_tL^2_x}\leq C \sqrt{{\alpha_n}E_0}T\\
		\rightarrow&\;0
	\end{align*}
	by Korn's inequality and the energy inequality for the Navier-Stokes equations. Note also that
	\begin{align*}
		(\rho_n)^{\gamma}\mathbb{E}_d:\nabla\varphi=(\rho_n)^{\gamma}\diverg\varphi.
	\end{align*}
	Thus,
	\begin{align*}
		&\;\int\limits_{0}^{T}\int\limits_{\Omega}^{}\rho_n u_n\cdot\partial_t\varphi+\rho_n u_n\otimes u_n:\nabla\varphi+(\rho_n)^{\gamma}\diverg\varphi\dx\dt\\
		\rightarrow&\;\int\limits_{0}^{T}\int\limits_{\Omega}^{}\langle\nu_{t,x},\rho u\rangle\cdot\partial_t\varphi+\langle\nu_{t,x},\rho u\otimes u\rangle:\nabla\varphi+\langle\nu_{t,x},\rho^{\gamma}\rangle\diverg\varphi\dx\dt+\int\limits_{0}^{T}\int\limits_{\Omega}^{}\nabla\varphi:\dd m_t\dt
	\end{align*}
	for all $\varphi\in C_c^{\infty}([0,T)\times \Omega;\R^d)$. This implies iv).
	
	Finally, the energy inequality vi) follows from the above discussion by taking traces of the terms $\rho_n u_n\otimes u_n$ and $(\rho_n)^{\gamma}\mathbb{E}_d$ with the correct coefficients.
\end{proof}
\subsection{Weak-Strong Uniqueness}
Although measure-valued solutions are a very weak notion of solution, they satisfy a certain uniqueness property. Clearly, there exist initial data with non-unique measure-valued solutions. Even weak solutions are already not unique. However, if the initial data gives rise to a classical solution, then every dissipative measure-valued solution corresponding to this initial data is already equal to the classical solution. This is called \textit{weak-strong uniqueness}. As a consequence, initial data giving rise to laminar flow cannot generate turbulent flow.

For simplicity, we restrict ourselves to the torus $\T^d$ as the space domain.
\begin{theorem}
	Let the initial data $u_0$ give rise to a strong incompressible solution $U\in C^1([0,T]\times \T^d)$ and an incompressible dissipative measure-valued solution $(\nu,m,D)$.
	
	Then $(\nu,m,D)=\left(\delta_{U},0,0\right)$.
\end{theorem}

\begin{theorem}
	Let the initial data $(\rho_0,u_0)$ give rise to a strong compressible solution $(r,U)\in C^1([0,T]\times \T^d)$, with $r>0$ everywhere, and also to a compressible dissipative measure-valued solution $(\nu,m,D)$.
	
	Then $(\nu,m,D)=\left(\delta_{(r,U)},0,0 \right)$.
\end{theorem}
We only prove the compressible case, since the proof in the incompressible case is simpler and follows from the same arguments. For a proof of the incompressible case, see e.g.~\cite{BDS11} or \cite{W18}.

As the astute reader will find, the exclusion of vacuum in the strong solution ($r>0$) allows us to raise $r$ to a negative power in the proof. Generally, weak-strong uniqueness is not known to hold in the presence of vacuum, but specific appearances of vacuum can be tolerated~\cite{GJW}.

For a strong solution $(r,U)$ with initial data $(r_0,U_0)$ and a measure-valued solution $(\nu,m,D)$ with initial data $(\rho_0,u_0)$, define the \textit{relative energy} as
\begin{align*}
	&\;E_{\operatorname{rel}}(\nu,m,D|r,U)(t)\\
	:=&\,\int\limits_{\T^d}^{}\left\langle\nu_{t,x},\frac{1}{2}|u-U(t,\cdot)|^2+\frac{1}{\gamma-1}\big(\rho^{\gamma}-r^{\gamma}\big)-\frac{\gamma}{\gamma-1}r^{\gamma-1}(\rho-r)\right\rangle\dx+D(t)
\end{align*}
for a.e.~$t\in (0,T)$ and at $t=0$ define
\begin{align*}
	E_{\operatorname{rel}}(\nu,m,D|r,U)(0):=\int\limits_{\T^d}^{}\frac{1}{2}|u_0-U_0|^2+\frac{1}{\gamma-1}\big(\rho_0^{\gamma}-r_0^{\gamma}\big)-\frac{\gamma}{\gamma-1}r_0^{\gamma-1}(\rho_0-r_0)\dx.
\end{align*}
By the strong convexity of $\rho\mapsto\rho^{\gamma}$, we directly infer the weak-strong uniqueness property from the concept of \textit{weak-strong stability}, which we formulate as follows.
\begin{theorem}\label{thm: weak-strong mvs}
	Let the initial data $(\rho_0,u_0)$ give rise to a compressible dissipative measure-valued solution $(\nu,m,D)$ and let the initial data $(r_0,U_0)$ give rise to a strong compressible solution $(r,U)\in C^1([0,T]\times \T^d)$.\\
	Then there exists a constant $C>0$ such that
	\begin{align*}
		E_{\operatorname{rel}}(\nu,m,D|r,U)(t)\leq e^{CT \|\nabla U\|_{C([0,T]\times\T^d)}} E_{\operatorname{rel}}(\nu,m,D|r,U)(0)
	\end{align*}
	for a.e.~$t\in(0,T)$.
\end{theorem}
\begin{proof}
	We follow the proof given in \cite{FGSW17}.
	
	As a preparation we test the continuity equation with $\frac{1}{2}|U|^2$ and $\frac{\gamma}{\gamma-1} r^{\gamma-1}$ multiplied with an appropriate cut-off function in time to obtain
	\begin{align*}
		\int\limits_{0}^{\tau}\int\limits_{\T^d}^{}\langle\nu,\rho\rangle U\cdot \partial_t U+\langle\nu,\rho u\rangle\cdot \nabla U\cdot U\dx\dt=\int\limits_{\T^d}^{}\langle \nu,\rho\rangle(\tau)\frac{1}{2}|U|^2(\tau)\dx-\int\limits_{\T^d}^{}\rho_0\frac{1}{2}|U_0|^2\dx
	\end{align*}
	and
	\begin{align*}
		&\;\int\limits_{0}^{\tau}\int\limits_{\T^d}^{}\langle\nu,\rho\rangle \gamma r^{\gamma-2} \partial_t r+\langle\nu,\rho u\rangle\cdot\gamma r^{\gamma-2}\nabla r\dx\dt\\
		=&\;\int\limits_{\T^d}^{}\langle \nu,\rho\rangle(\tau)\frac{\gamma}{\gamma-1}r^{\gamma-1}(\tau)\dx-\int\limits_{\T^d}^{}\rho_0\frac{\gamma}{\gamma-1}r_0^{\gamma-1}\dx
	\end{align*}
	for a.e.~$\tau\in (0,T)$. Moreover, testing the momentum equation with $U$ yields
	\begin{align*}
		&\;\int\limits_{0}^{\tau}\int\limits_{\T^d}^{}\langle\nu,\rho u\rangle\cdot \partial_t U+\langle\nu,\rho u\otimes u\rangle:\nabla U+\langle\nu,\rho^{\gamma}\rangle\diverg U\dx\dt +\int\limits_{0}^{\tau}\int\limits_{\T^d}^{}\nabla U:\dd m_t\dt\\
		=&\;\int\limits_{\T^d}^{}\langle\nu,\rho u\rangle(\tau)\cdot U(\tau)\dx-\int\limits_{\T^d}^{}\rho_0u_0\cdot U_0\dx.
	\end{align*}
	Using these identities and the energy inequality for $\nu$, we obtain
	\begin{align*}
		&\;E_{\operatorname{rel}}(\nu,m,D|r,U)(\tau)\\
		\leq &\;E_{\operatorname{rel}}(\nu,m,D|r,U)(0)-\int\limits_{0}^{\tau}\int\limits_{\T^d}^{}\nabla U:\dd m_t\dt+\int\limits_{\T^d}^{}r^{\gamma}-r_0^{\gamma}\dx\\
		&\;-\int\limits_{0}^{\tau}\int\limits_{\T^d}^{}\langle\nu, \rho u\rangle\cdot \partial_t U+\langle\nu, \rho u\otimes u\rangle:\nabla U+\langle\nu,\rho^{\gamma}\rangle\diverg U\dx\dt\\
		&\;-\int\limits_{0}^{\tau}\int\limits_{\T^d}^{}\langle\nu,\rho \rangle \left(\gamma r^{\gamma-2}\partial_t r-U\cdot\partial_t U\right)+\langle\nu, \rho u\rangle \cdot\left(\gamma r^{\gamma-2}\nabla r-\nabla U\cdot U\right)\dx\dt\\
		=&\;E_{\operatorname{rel}}(\nu,m,D|r,U)(0)+C\|\nabla U\|_{C([0,T]\times \T^d)}\int\limits_{0}^{\tau}D(t)\dt\\
		&\;-\int\limits_{0}^{\tau}\int\limits_{\T^d}^{}\langle\nu,\rho^{\gamma}-r^{\gamma}-\gamma r^{\gamma-1}(\rho-\gamma)\rangle\diverg U\dx\dt\\
		&\;-\int\limits_{0}^{\tau}\int\limits_{\T^d}^{}\langle\nu, \rho u\rangle\cdot \partial_t U+\langle\nu, \rho u\otimes u\rangle:\nabla U\dx\dt\\
		&\;-\int\limits_{0}^{\tau}\int\limits_{\T^d}^{}\langle\nu,\rho U\rangle \cdot\left(-\gamma r^{\gamma-2} \nabla r-\partial_t U\right)+\langle\nu, \rho u\rangle \cdot\left(\gamma r^{\gamma-2}\nabla r-\nabla U\cdot U\right)\dx\dt,
	\end{align*}
	where in the last equality we also used $\partial_t r+r\diverg U+U\cdot\nabla r=0$. Now, by the momentum equation we have
	\begin{align*}
		\partial_t U+(U\cdot\nabla)U+\gamma r^{\gamma-2}\nabla r=0
	\end{align*}
	for $U$. Hence, we get for a.e.~$\tau\in (0,T)$
	\begin{align*}
		&\;E_{\operatorname{rel}}(\nu,m,D|r,U)(\tau)\\
		\leq &\;E_{\operatorname{rel}}(\nu,m,D|r,U)(0)+C\|\nabla U\|_{C([0,T]\times \T^d)}\int\limits_{0}^{\tau}D(t)\dt\\
		&\;-\int\limits_{0}^{\tau}\int\limits_{\T^d}^{}\langle\nu,\rho^{\gamma}-r^{\gamma}-\gamma r^{\gamma-1}(\rho-\gamma)\rangle\diverg U\dx\dt\\
		&\;+\int\limits_{0}^{\tau}\int\limits_{\T^d}^{}\langle\nu,\rho u\rangle\cdot \nabla U\cdot U+\langle\nu,\rho U\cdot \nabla U\cdot u\rangle\\
		&\;-\int\limits_{0}^{\tau}\int\limits_{\T^d}^{}\langle\nu, \rho u\otimes u\rangle:\nabla U+\langle\nu,\rho U\rangle\cdot\nabla U\cdot U\dx\dt\\
		=&\;E_{\operatorname{rel}}(\nu,m,D|r,U)(0)+C\|\nabla U\|_{C([0,T]\times \T^d)}\int\limits_{0}^{\tau}D(t)\dt\\
		&\;-\int\limits_{0}^{\tau}\int\limits_{\T^d}^{}\langle\nu,\rho^{\gamma}-r^{\gamma}-\gamma r^{\gamma-1}(\rho-\gamma)\rangle\diverg U\dx\dt\\
		&\;-\int\limits_{0}^{\tau}\int\limits_{\T^d}^{}\langle\nu,\rho(u-U)\cdot\nabla U\cdot (u-U)\rangle\dx\dt\\
		\leq &\;E_{\operatorname{rel}}(\nu,m,D|r,U)(0)+C\|\nabla U\|_{C([0,T]\times \T^d)}\int\limits_{0}^{\tau}E_{\operatorname{rel}}(\nu,m,D|r,U)(t)\dt
	\end{align*}
	for some $C>0$ only depending on the dimension $d$ and the constant in \Cref{def: mvsCE} iii). The claim now follows by Grönwall's inequality.
\end{proof}

\subsection{Measure-Valued Singular Limits}
For the existence of measure-valued solutions, we have seen that the viscosity limit corresponding to fixed initial data generates a measure-valued solution. Moreover, every Young measure can be generated by some sequence of measurable functions. So, a natural question to ask is whether every measure-valued solution might come from a sequence of (approximate) weak solutions, e.g.~a vanishing viscosity sequence. Other possibilities might be the limit of numerical schemes, low Mach limits, or limits of weak solutions. Note that it is of course not included in the definition of measure-valued solutions that they are generated by approximate solutions.

In the incompressible case, it is a surprising and highly non-trivial fact that every measure-valued solution comes from a sequence of weak solutions, cf.~\cite{SW12} for a proof.
\begin{theorem}\label{thm: incompressiblemvs}
	Let $\nu$ be an incompressible dissipative measure-valued solution on $\R^d$ with initial data $u_0\in L^2(\R^d)$. Then there exists a sequence of energy admissible weak solutions $(u_n)_{n\in\N}$ with initial data $(u_n^0)_{n\in\N}$ such that $u^0_n\rightarrow u_0$ in $L^2(\R^d)$ and such that $(u_n)_{n\in\N}$ generates $\nu$, i.e.
	\begin{align*}
		u_n\overset{Y}{\rightharpoonup}\nu.
	\end{align*}
\end{theorem}
In fact, this result of Sz{\'e}kelyhidi-Wiedemann uses a different formulation of concentration measures, see \cite{SW12}. But we do not want to go deeper into this topic.

The surprising consequence, anyway, of the above theorem is that incompressible measure-valued solutions and weak solutions contain essentially the same information although measure-valued solutions are an a priori much weaker notion of solution, since only the first and second moment of the measure are constrained.

So, what about the compressible situation? In fact, the case of the compressible Euler system is quite different and also more complicated. As a preparation we have the following meta-result.
\begin{theorem}\label{thm:nongenerablemvs}
	For space dimension $d\geq 2$, not every oscillation measure-valued solution of the compressible Euler equations comes from a sequence of weak solutions.
\end{theorem}
This has first been observed in \cite{CFKW17}, where a constant measure-valued solution consisting of two Dirac measures is constructed such that it cannot be generated by weak solutions. The idea of \cite{CFKW17} has been developed further in \cite{GW21} to obtain the result that there exists even deterministic initial data, which evolves classically up to some finite time and then bursts into actually infinitely many oscillation measure-valued solutions that cannot be generated by weak solutions.\\
The idea of the proof for both versions of the above theorem relies on a compensated compactness rigidity argument in the spirit of Ball and James~\cite{BJ87}. For this, the rank-one connectedness in \cite{BJ87} corresponds to the \textit{wave-cone connectedness} of the underlying Young measures. Here, we consider the calculus of variations framework of \textit{linear homogeneous differential operators} $\mathcal{A}$ of \textit{order} $k\in\N$, i.e.
\begin{align*}
	\mathcal{A}=\sum\limits_{|\alpha|=k}^{}A^{\alpha}\partial_{\alpha}
\end{align*}
with $A^{\alpha}$ constant coefficient matrices acting on the state space $\R^N$. This framework was pioneered by Tartar, cf.~\cite{T79}. The corresponding Fourier symbol of $\mathcal{A}$ is defined for $\xi\in\R^d$ as
\begin{align*}
	\mathbb{A}(\xi)=\sum_{|\alpha|=k}^{}\xi_{\alpha}A^{\alpha}.
\end{align*}
The \textit{wave-cone} of $\mathcal{A}$ is defined as
\begin{align*}
	\Lambda_{\mathcal{A}}=\underset{|\xi|=1}{\bigcup}\ker\mathbb{A}(\xi).
\end{align*}
Moreover, two vectors $z_1,z_2\in \R^N$ are called \textit{wave-cone connected} if
\begin{align*}
	z_1-z_2\in\Lambda_{\mathcal{A}}.
\end{align*}
Now the compressible Euler system can be relaxed to a linear homogeneous system of PDEs of order one. The corresponding linear operator is obtained via simple substitution of the non-linear terms by new variables. We call this the \textit{linearly relaxed Euler system}
\begin{align*}
	\partial_t m+\diverg M+\nabla Q&= 0,\\
	\partial_t\rho + \diverg m&= 0
\end{align*}
for the variables $(\rho,m,M,Q)\in [0,\infty)\times \R^d\times S_0^d\times \R$, where $S_0^d$ denotes the symmetric $(d\times d)$-matrices with zero trace. The relaxed system now directly corresponds to a linear homogeneous differential operator $\mathcal{A}_E$ of order one. The transition from the non-linear to the linearized system is achieved through the \textit{lift map}
\begin{align*}
	\Theta\colon (\rho,m)\mapsto \left(\rho,m,\frac{m\otimes m}{\rho}-\frac{|m|^2}{\rho d}\mathbb{E}_d,\rho^{\gamma}+\frac{|m|^2}{\rho d}\right),
\end{align*}
where $\mathbb{E}_d$ is the unit matrix in $d$ dimensions. Here, for convenience, we used the momentum $m=\rho u$ as a primitive variable instead of the velocity $u$. Anyway, this lift map $\Theta$ encodes the non-linear pointwise constraints that together with the linear operator $\mathcal{A}_E$ yields the non-linear compressible Euler system.

Another consequence of considering the linearized Euler system $\mathcal{A}_E$ is the following necessary condition, see \cite{CFKW17} and \cite{GW22}. This follows from the seminal work of Fonseca-Müller \cite{FM99} which is a generalization of the work of Kinderlehrer-Pedregal \cite{KP91,KP94}. In fact, one may also infer Theorem~\ref{thm:nongenerablemvs} from Theorem~\ref{thm:necessarycondition}.
\begin{theorem}\label{thm:necessarycondition}
	Suppose $\nu$ is a Young measure which is generated by a uniformly bounded sequence of energy admissible weak solutions $(\rho_n,m_n)_{n\in\N}$ of the compressible Euler system over $(0,T)\times \T^d$. Further, assume that the initial data satisfy $(\rho_n^0,m_n^0)\rightarrow (\rho_0,m_0)$ in $L^{\gamma}\times L^2$ and assume that there exists $\eta>0$ such that $\rho_n\geq \eta$ for all $n\in \N$.
	
	Then $\nu$ is a dissipative oscillation measure-valued solution with initial data $(\rho_0,m_0)$ and the Jensen-type condition
	\begin{align*}
		\left\langle\Theta_{\sharp}\nu,f\right\rangle\geq Q_{\mathcal{A}_E}f\left(\left\langle\Theta_{\mathcal{A}_E}\nu,\operatorname{id}\right\rangle \right)
	\end{align*}
	holds on a set of full measure and for all $f\in C(\R^N)$.
\end{theorem}
We need to clarify some notation:

The \textit{lifted measure} $\Theta_{\sharp}\nu$ is defined as the pushforward measure
\begin{align*}
	\left\langle\Theta_{\sharp}\nu,f\right\rangle=\langle\nu,f\circ \Theta\rangle\text{ for }f\in C(\R^N).
\end{align*}
As a consequence, $\nu$ is an oscillation measure-valued solution of the compressible Euler system if and only if the corresponding lifted measure satisfies the linear constraint
\begin{align*}
	\mathcal{A}_E\langle\Theta_{\sharp}\nu,\operatorname{id}\rangle=0.
\end{align*}
Moreover, for any linear homogeneous differential operator $\mathcal{A}$ we say that a function $f\in C(\R^N)$ is $\mathcal{A}$-\textit{quasiconvex} if for all $z\in \R^N$
\begin{align*}
	\int\limits_{\T^d}^{}f(z+w(x))\dx\geq f(z)
\end{align*}
holds for all $w\in C^{\infty}(\T^d,\R^N)$ with $\mathcal{A}w=0$ and $\int\limits_{\T^d}^{}w(x)\dx=0$. Note that for $\mathcal{A}=0$ this corresponds to $f$ being just convex.

For a given function $f\in C(\R^N)$, the largest $\mathcal{A}$-quasiconvex function below $f$ is the $\mathcal{A}$-\textit{quasiconvex envelope} $Q_{\mathcal{A}}f$ defined by
\begin{align*}
	Q_{\mathcal{A}}f(z)=\inf\left\{\avint_{\T^d}f(z+w(x))\dx\,:\, w\in C^{\infty}(\T^d)\cap \ker\mathcal{A},\ \int\limits_{\T^d}^{}w(x)\dx=0 \right\}.
\end{align*}
Let $\mathcal{B}$ be a \textit{potential operator} for $\mathcal{A}$, i.e.~the linear homogeneous differential operator $\mathcal{B}$ of order $l\in\N$ satisfies
\begin{align*}
	\ker\mathbb{A}(\xi)=\operatorname{image}\mathbb{B}(\xi)\text{ for }\xi\in\R^d\backslash\{0\}.
\end{align*}
Then for $q>0$ and $z\in\R^N$ we define the \textit{truncated quasiconvex envelope} $Q_{\mathcal{B}}^qf$ by
\begin{align*}
	Q_{\mathcal{B}}^qf(z)=\inf\left\{\int\limits_{(0,1)^d}^{}f(z+\mathcal{B}\varphi(x))\dx\,:\, \varphi\in C_c^{\infty}((0,1)^d),\ \|D^l\varphi\|_{L^{\infty}}\leq q \right\}.
\end{align*}
A direct consequence of the above definitions is that for all $f\in C(\R^N)$ and $q>0$ it holds that
\begin{align*}
	Q_{\mathcal{A}}f\leq Q_{\mathcal{B}}^qf\leq f.
\end{align*}
See, e.g., \cite{GW22} and \cite{FM99} for more on (truncated) quasiconvex envelopes.

The above Theorem~\ref{thm:necessarycondition} suggests the following selection criterion for measure-valued solutions:

\emph{A measure-valued solution may be discarded as unphysical if it is not generated by a sequence of weak solutions.}

In fact, the selection via singular limits is a common way to identify unphysical solutions, see \cite{BTW12} for the vanishing viscosity case and \cite{G22} for the case of low Mach number limits.

Let us quickly review the incompressible situation. One can also derive a necessary Jensen condition in this case. However, since for the corresponding linear operator every quasiconvex function is already convex, this Jensen condition is trivially fulfilled by the classical Jensen inequality, cf.~Remark~4 in \cite{CFKW17}. Therefore, we do not get a contradiction from the fact that every incompressible measure-valued solution is generated by weak solutions.

We have learned that compressible measure-valued solutions which can be generated by weak solutions necessarily need to satisfy a Jensen-type condition. But what about sufficient conditions? It turns out that a related Jensen condition among some other structural conditions only on the measure are sufficient. This is stated in the following theorem from \cite{GW22}.
\begin{theorem}
	Let $T>0$ and $d\geq 2$. Let $\nu$ be a dissipative oscillation measure-valued solution of the compressible Euler system on $(0,T)\times \T^d$ with initial data $(\rho_0,m_0)\in L^{\infty}(\T^d)$ and adiabatic exponent $\gamma=1+\frac{2}{d}$. Suppose $\nu$ satisfies the following conditions:
	\begin{itemize}
		\item There exists $R>0$ such that
		\begin{align*}
			\langle \Theta_{\sharp}\nu,f\rangle\geq Q_{\mathcal{B}_E}^Rf(\langle\Theta_{\sharp}\nu,\operatorname{id}\rangle)
		\end{align*}
		holds on a set of full measure for all $f\in C([0,\infty)\times \R^d\times S_0^d\times \R)$.
		\item There exists $\eta>0$ such that
		\begin{align*}
			\operatorname{supp}\left((\Theta_{\sharp}\nu)_{(t,x)}\right)\subset \left\{(\rho,m,M,Q)\,:\, \rho\geq \eta\text{ and }Q=\left\langle\nu_{(t,x)},\frac{1}{d}\frac{|m|^2}{\rho}+\rho^{\gamma}\right\rangle \right\} 
		\end{align*}
		holds for a.e.~$(t,x)\in (0,T)\times \T^d$.
		\item There exist $w\in W^{2,\infty}((0,T)\times\T^d)$ and $\sigma\in C([0,T]\times \T^d)$ such that the barycenter of the lifted measure satisfies
		\begin{align*}
			\langle\Theta_{\sharp}\nu,\operatorname{id}\rangle=\sigma+\mathcal{B}w.
		\end{align*}
		\item The map $(t,x)\mapsto \left\langle\nu_{(t,x)},\frac{1}{d}\frac{|m|^2}{\rho}+\rho^{\gamma}\right\rangle$ is continuous on $[0,T]\times \T^d$.
	\end{itemize}
	Then $\nu$ is generated by a uniformly bounded sequence of energy admissible weak solutions $(\rho_n,m_n)_{n\in\N}$ such that for all $n\in \N$ it holds that
	\begin{align*}
		\rho_n&\geq \tilde{\eta},\\
		\|\rho_n(0,\cdot)-\rho_0\|_{L^{\gamma}}&\leq \frac{1}{n},\\
		\|m_n(0,\cdot)-m_0\|_{L^2}&\leq \frac{1}{n}
	\end{align*}
	for some fixed $\tilde{\eta}>0$.
\end{theorem}
Here, $\mathcal{B}_E$ denotes the second order potential of the relaxed Euler system $\mathcal{A}_E$ which is constructed in \cite{G21}.

We want to give here only the idea of the proof:

Starting from the Jensen condition for the lifted measure, we obtain a sequence of weak solutions of the relaxed system generating $\Theta_{\sharp}\nu$. This follows from a result in the spirit of the Fonseca-Müller characterization from \cite{G22}. A truncation method from \cite{G21} can be used to obtain a generating sequence that is suitable for applying convex integration methods. The particular convex integration result we use is from \cite{DSW21}. In this way, the sequence of relaxed solutions generating $\Theta_{\sharp}\nu$ gives a sequence of true energy admissible weak solutions generating the measure $\nu$. This means, by convex integration we can transition from the linear back to the non-linear regime.

Now also the condition of space dimension $d\geq 2$ is justified, as convex integration methods do not work for $d=1$. Moreover, the condition $\gamma=1+\frac{2}{d}$ stems also from the usage of the specific compressible convex integration scheme we use. The main advantage of this case is that the generalized pressure $Q$ equals the energy density $\frac{1}{2}\frac{|m|^2}{\rho}+\frac{1}{\gamma-1}\rho^{\gamma}$ up to the constant factor $\frac{d}{2}$, which is crucial for the proofs. Note, however, that this particular choice of $\gamma$ corresponds physically to the case of monoatomic gases.

Let us briefly comment on the gap between the above necessary and sufficient conditions for generating compressible measure-valued solutions by weak solutions. When we ignore energy admissibility, the only true gap between the necessary and sufficient Jensen and barycenter conditions arises from the breakdown of the Calder\'on-Zygmund inequality in the case of $L^{\infty}$, which results in the different quasiconvex envelopes we used. It is unclear for now how to treat this very subtle point.

If we consider the conditions that lead to energy admissibility, we observe that the assumption $Q=\left\langle\nu,\frac{1}{d}\frac{|m|^2}{\rho}+\rho^{\gamma}\right\rangle$ is way too restrictive as a general dissipative measure-valued solution may violate this. For example take a solution consisting of two Dirac measures supported at weak solutions with different energies. Moreover, the continuity condition on the energy of $\nu$ apparently is also very restrictive. However, this is needed for technical reasons in the proof, which cannot be overcome in an obvious way.

We want to give also a very heuristic argument of why the incompressible and compressible situations are different:

It is easier to find true solutions generating the measure if the convex integration scheme is more flexible. This, in turn, corresponds to a large wave-cone, since we can find solutions of the relaxed system more easily then. A large wave-cone, however, implies that more functions $w$ satisfy $\mathcal{A}w=0$, and hence qualify as test functions in the definition of $\mathcal{A}$-quasiconvexity. Thus, the difference between a function and its corresponding $\mathcal{A}$-quasiconvex envelope is larger. So, the $\mathcal{A}$-quasiconvex envelope is relatively small.

Note that for the incompressible system the associated wave-cone is large compared to the wave-cone corresponding to the compressible system. Therefore the Jensen condition is easier to satisfy in the incompressible case, in fact, it is void.


\section{Statistical Notions of Solution}\label{sec: stats sols}

\subsection{Definitions and Comparisons Between Different Notions}\label{subsec: definitions}
In the following, we work on the two-dimensional torus $\T^2$ and some finite but arbitrary time interval $[0,T]$, and denote by $\CalU$ the set of all weak solutions of the incompressible Euler equations in $C([0,T];H_w)$ which satisfy the energy inequality
\begin{equation}\label{eq: energy ineq EE}
\frac{1}{2}\int_{\T^2} |u(x,t)|^2\dx \leq \frac{1}{2}\int_{\T^2} |u(x,0)|^2\dx
\end{equation}
for almost every $0 \leq t \leq T$. Here, $H_w$ is the set of $L^2(\T^2;\R^2)$-vector fields that are weakly divergence-free, equipped with the weak topology, so that $t\mapsto \int_{\T^2}u(x,t)\cdot v(x)\dx$ is continuous in $[0,T]$ for each $v\in L^2(\T^2;\R^2)$. In fact, this implies -- by weak lower semicontinuity of the norm -- that~\eqref{eq: energy ineq EE} is satisfied for {\em every} $t\in[0,T]$.

We introduce the time evaluation mappings $\lbrace \Pi_t\rbrace_{0 \leq t \leq T}$, where $\Pi_t\colon C([0,T];H_w) \to H_w, u \mapsto u(t)$.

For the considerations to follow, we note the general fact that on Polish spaces such as $H$, the Borel-$\sigma$-algebra generated by the weak and strong topologies coincide. We will therefore for purely measure-theoretic considerations and notations not distinguish between $H$ and $H_w$.

\begin{definition}[Trajectory statistical solutions]\label{def: stats sol trajectory}
Let $\overline{\mu}$ be a Borel probability measure on $H$ such that
\[\int_{H} \|u_0\|_{L^2}^2\dd\overline{\mu}(u_0) < \infty.\]
A {\em trajectory statistical solution} of the (2D incompressible) Euler equations or {\em $\CalU$-trajectory statistical solution} with initial distribution $\overline{\mu}$ is a Borel probability measure $\mu$ on $C([0,T];H_w)$ such that
\begin{enumerate}[i)]
\item there exists a Borel measurable subset $\CalV \subset \CalU$ such that $\mu(\CalV) = 1$;
\item the initial data is attained in the sense that the pushforward measure ${\Pi_0}_\sharp\mu$ is equal to $\overline{\mu}$;
\item $\mu$ satisfies an energy inequality in the sense that for every $0 \leq t \leq T$
\begin{equation}\label{eq: energy ineq traj stats sols}
\int_{C([0,T];H_w)} \|u(t)\|_{L^2}^2\dd\mu(u) \leq \int_{H}\|u_0\|_{L^2}^2\dd\overline{\mu}(u_0).
\end{equation}
\end{enumerate}
\end{definition}

For the definition of phase space statistical solutions, we need an appropriate class of test functionals $\CalT$. Motivated by \cite{FMRT01,FRT13,BMR16}, we let $\CalT$ be the set of functionals $\Phi\colon H \to \R$ such that for some $k\in \N$ and $L>2$, there exists $\phi \in C^1_c(\R^k)$ and $g_1,...,g_k \in H^{L}(\T^2;\R^2) \cap H \eqqcolon H^L$ such that
\[\Phi(u) = \phi((u,g_1)_{L^2},...,(u,g_k)_{L^2})\]
for all $u \in H$. In particular $\CalT \subset C_b(H_w)$.
Such functionals $\Phi$ are also Fr{\'e}chet-differentiable by the chain rule and the Fr{\'e}chet derivative $\Phi'(u)$ at some $u \in H$ can be identified with the function $\sum_{j=1}^k \partial_j \phi((u,g_1)_{L^2},...,(u,g_k)_{L^2})g_j$ in $H^{L}$.

As $L > 2$, we note that as we consider the 2D case, we have the embedding $H^L \hookrightarrow C^{1,\beta}(\T^2;\R^2)$ for some $0 < \beta < 1$ so that $\nabla \Phi'(u) \in C^{0,\beta}(\T^2;\R^2) \subset L^\infty(\T^2;\R^2)$.
For $u \in H$ and $g_1,...,g_k \in H^L$, we will also frequently use the notation $(u,g)_{L^2}^k \coloneqq ((u,g_1)_{L^2},...,(u,g_k)_{L^2}) \in \R^k$.\\
The functionals in $\CalT$ are called \textit{cylindrical test functionals} since to each Borel measure $\mu$ on $H$ one may uniquely associate the generalized moments $\lbrace \int_H \Phi(u)\dd\mu(u) : \Phi \in \CalT \rbrace.$ This may for instance be proved using the Stone-Weierstrass theorem \cite[Remark 2.7]{BMR16} or a monotone class argument.

In light of \eqref{eq: Foias-Liouville}, which was formally derived in the introduction, we make the following definition.

\begin{definition}[Statistical solutions in phase space]\label{def: stats sol phase space}
Let $\overline{\mu}$ be a Borel probability measure on $H$ such that
\[\int_{H} \|u_0\|_{L^2}^2\dd\overline{\mu}(u_0) < \infty.\]
A {\em phase space statistical solution} of the (2D incompressible) Euler equations with initial distribution $\overline{\mu}$ is a family $\lbrace \mu_t \rbrace_{0 \leq t \leq T}$ of Borel probability measures on $H$ such that $\mu_0 = \overline{\mu}$ and
\begin{enumerate}[i)]
\item for every $\varphi \in C_b(H_w)$
\begin{equation}\label{eq: phase space stats sol measurability}
t \mapsto \int_{H} \varphi(u)\dd\mu_t(u)
\end{equation}
is a Lebesgue measurable mapping on $[0,T]$;
\item the energy inequality 
\begin{equation}\label{eq: phase space energy inequ}
\int_{H} \|u\|_{L^2}^2\dd\mu_t(u) \leq \int_{H} \|u\|_{L^2}^2\dd\overline{\mu}(u)
\end{equation}
holds for a.e. $0 \leq t \leq T$;
\item for any $\Phi \in \CalT$ and every $0 \leq t' \leq t \leq T$, $\lbrace \mu_t \rbrace_{0 \leq t \leq T}$ satisfies the Foia\cs-Liouville equation
\begin{equation}\label{eq: foias-liouville def}
\int_H \Phi(u)\dd\mu_t(u) = \int_H \Phi(u)\dd\mu_{t'}(u) + \int_{t'}^t \int_H (\nabla\Phi'(u),u \otimes u)_{L^2}\dd\mu_s(u)\dd s.
\end{equation}  
\end{enumerate}
\end{definition}

\begin{remark}
We have chosen here $H_w$ as phase space for our notion of phase space statistical solutions and also defined trajectory statistical solutions as measures on $C([0,T];H_w)$. As in \cite{WW23}, it would likewise be possible and natural to instead use some negative order Sobolev space $H^{-s}(\T^2;\R^2), s > 2$. We chose $H_w$ here, as it is more natural for the comparison to moment-based statistical solutions that shall be introduced later.
\end{remark}

The following theorem shows in what way trajectory statistical solutions correspond to phase-space statistical solutions. It can be proved as in \cite[Chapter 3]{FRT13} (see also \cite{BMR16, WW23}).

\begin{theorem}\label{thm: traj stats sol -> phase space stats sol}
Let $\overline{\mu}$ be a Borel probability measure on $H$ and let $\mu$ be a trajectory statistical solution of the Euler equations with initial distribution $\overline{\mu}$. Then the family of pushforward measures $\lbrace\mu_t := {\Pi_t}_{\sharp}\mu\rbrace_{0 \leq t \leq T}$ is a statistical solution in phase space of the Euler equations with initial distribution $\overline{\mu}$.
\end{theorem}

At this point, let us also describe the Hopf statistical equation more rigorously. As in our considerations for deriving the Foia\cs-Liouville equation \eqref{eq: Foias-Liouville}, given a Borel probability measure on $H$, we assume the existence of solution operators $\lbrace S_t\rbrace_{0 \leq t \leq T}$ such that $t \mapsto S_tu_0 \in C([0,T];H_w)$ is a weak solution of the Euler equations with initial data $u_0 \in H$. Then, the evolution of $\overline{\mu}$ is naturally described by the pushforward measures $\mu_t \coloneqq {S_t}_\sharp\overline{\mu}$, $0 \leq t \leq T$, on $H_w$. Their characteristic functionals are given by
\[\F(t,v) = \int_{H} \exp(i(u,v)_{L^2})\dd\mu_t(u) = \int_H \exp(i(S_tu_0,v)_{L^2})\dd\overline{\mu}(u_0)\]
and the evolution of $\F$ can be described as follows
\begin{equation}\label{eq: Hopf equ derivation}
\begin{split}
\partial_t \F(t,v) &= \int_H i(\partial_tS_tu_0,v)_{L^2}\exp(i(S_tu_0,v)_{L^2})\dd\overline{\mu}(u_0)\\
&= \int_H i(S_tu_0 \otimes S_tu_0,\nabla v)_{L^2}\exp(i(S_tu_0,v)_{L^2})\dd\overline{\mu}(u_0)\\
&= \int_H i(u\otimes u, \nabla v)_{L^2}\exp(i(u,v)_{L^2})\dd\mu_t(u)
\end{split}
\end{equation}
for a.e. $0 \leq t \leq T$ and $v \in H^{L}$. The right-hand side is well-defined if $\mu_t$ has finite second moment with respect to $\|\cdot\|_{L^2}$. We denote the set of characteristic functionals of Borel probability measures on $H$ with this property by $\CalC$. We then introduce an operator $B$ on $\CalC$, such that $B(\chi)\colon H \to H^{-L}\coloneqq (H^L)'$ is given by
\[(B(\chi)w,v)_{H^{-L},H^{L}} = \int_{H} i(u \otimes u,\nabla v)_{L^2}\exp(i(u,w)_{L^2})\dd\mu(u)\]
for all $w \in H$ and $v \in H^{L}$ and $\chi \in \CalC$ being the characteristic functional of a measure $\mu$.\\
Then we may rewrite \eqref{eq: Hopf equ derivation} as
\begin{equation}\label{eq: Hopf functional equ}
\partial_t \F(t,v) = (B\F(t,v),v)_{H^{-L},H^{L}},
\end{equation}
which we call the \textit{Hopf statistical equation in functional form} of the (2D incompressible) Euler equations.
Alternatively, closer to Hopf's original formal considerations, one can study the \textit{coordinate form}. By this, we mean the description in terms of Fourier coefficients with respect to some fixed orthonormal basis of $H$, say $(\varphi_n)_{n\in\N} \subset C^\infty(\T^2) \cap H$ in our case. We denote the Fourier coefficients of some $v \in H$ by $\hat{v}_n \coloneqq (v,\varphi_n)_{L^2}$, $n \in \N$, and more specifically, denote the Fourier coefficients of $S_tu_0$ by $w_n(t,u_0)$, so that $S_tu_0 = \sum_{n=1}^\infty w_n(t,u_0)\varphi_n$ and 
\[\F(t,v) = \int_H \exp\left( i \sum_{n=1}^\infty w_n(t,u_0)\hat{v}_n\right)\dd\overline{\mu}(u_0).\]
Then, from \eqref{eq: Hopf equ derivation}, we obtain
\begin{equation}
\begin{split}
\partial_t \F(t,v) &= \int_H i\sum_{j,k,m=1}^\infty a_{jkm}w_j(t,u_0)w_k(t,u_0)\hat{v}_m\exp\left( i \sum_{n=1}^\infty w_n(t,u_0)\hat{v}_n\right)\dd\overline{\mu}(u_0)\\
&= \sum_{j,k,m=1}^\infty -i a_{jkm}\hat{v}_m \partial_{\hat{v}_j\hat{v}_k}\F(t,v),
\end{split}
\end{equation}
where $a_{jkm} \coloneqq (\varphi_j \otimes \varphi_k, \nabla \varphi_m)_{L^2}, j,k,m \in \N$. Then
\begin{equation}\label{Hopf equ coordinate form}
\partial_t \F(t,v) = \sum_{j,k,m=1}^\infty -ia_{jkm}\hat{v}_m \partial_{\hat{v}_j\hat{v}_k}\F(t,v)
\end{equation}
is called Hopf statistical equation in coordinate form of the (2D incompressible) Euler equations.

As is probably apparent here, it is quite inconvenient to give precise meaning to either the functional or the coordinate form of the Hopf statistical equation. Due to the similarity in the derivation of the Foia\cs-Liouville equation and statistical solutions in phase space, it might not come as a surprise that the latter actually lead to solutions of the Hopf statistical equation. We state the following result for our situation of the Euler equations, which is analogous to the considerations by Foia\cs\, in case of the Navier-Stokes equations \cite[Section 9, Theorem 1]{F73}, \cite[Sections V.1.3, V.2.4]{FMRT01}. See also \cite[Theorem 2]{LV78} by Ladyzhenskaya and Vershik for the coordinate version.

\begin{theorem}\label{thm: hopf equ existence}
Let $\overline{\mu}$ be a Borel probability measure on $H$ satisfying
\[\int_H \|u_0\|_{L^2}^2\dd\overline{\mu}(u_0) < \infty.\]
Suppose that $\lbrace \mu_t \rbrace_{0 \leq t \leq T}$ is a statistical solution of the Euler equations in phase space with initial distribution $\overline{\mu}$ and denote the characteristic functionals of $\overline{\mu}$ and $\lbrace \mu_t \rbrace_{0 \leq t \leq T}$ by $\overline{\F}$ and $\F = \lbrace \F(t,\cdot) \rbrace_{0\leq t \leq T}$ respectively.

Then $\F$ satisfies the functional version of the Hopf statistical equation in the sense that
\begin{equation}\label{eq: Hopf functional equ thm}
\F(t,v) = \overline{\F}(v) + \int_0^t (B\F(s,v),v)_{H^{-L},H^L}\dd s
\end{equation}
for all $v \in H^{L}$ and every $0 \leq t \leq T$.
\end{theorem} 

The main part in the proof of \Cref{thm: hopf equ existence}, which we are omitting here, is that for fixed $v \in H^{L}$, the class $\CalT$ of admissible functionals in the formulation of statistical solutions in phase-space can be extended to also include $u \mapsto \exp(i (u,v)_{L^2})$. This can be done by a standard truncation argument and a complexification of the involved spaces $H$ and $H^L$. 

For the third notion of statistical solution that we mentioned in the introduction, we first define \textit{correlation measures} as introduced in \cite{FLM17}, but only consider the $L^2$ setting. We introduce the sets of (time-dependent) Carath{\'e}odory functions
\begin{align*}
&\CalH_0^k(\T^2;\R^2) := L^1((\T^2)^k;C_0((\R^2)^k)),\\
 &\CalH_0^k([0,T],\T^2;\R^2) := L^1([0,T] \times (\T^2)^k;C_0((\R^2)^k)))
\end{align*} 
and their respective dual spaces
\begin{align*}
&\CalH_0^{k*}(\T^2;\R^2) := L^\infty_{\operatorname{w}}((\T^2)^k;\CalM((\R^2)^k))),\\
 &\CalH_0^{k*}([0,T],\T^2;\R^2) := L^\infty_{\operatorname{w}}([0,T] \times (\T^2)^k;\CalM((\R^2)^k)))
\end{align*}
for every $k \in \N$, where $\CalM$ denotes the set of all bounded Radon measures on the respective space. The subscript ${\operatorname{w}}$ denotes weak-* measurability as in \Cref{defmvs}.
\begin{definition}\label{def: corr meas}
A {\em correlation measure} is a collection $\corrM = (\nu^1,\nu^2,...)$ of maps $\nu^k \in \CalH_0^{k*}(\T^2;\R^2)$ satisfying the following properties:
\begin{enumerate}
\item Weak-* measurability: Every $\nu^k$ is a Young measure, i.e., for every $f \in C_0((\R^2)^k)$, the mapping
\[x\mapsto \langle \nu^k_x,f(\xi)\rangle := \int_{(\R^2)^k}f(\xi)\dd\nu_x^k(\xi)\]
on $(\T^2)^k$ is measurable. 
\item $L^2$-boundedness: $\int_{\T^2} \langle \nu_x^1,|\xi|^2\rangle \,dx < \infty$.
\item Symmetry: For every permutation $\pi$ of $\lbrace 1,...,k\rbrace$ and $f \in C_0((\R^2)^k)$, we have 
\begin{equation*}
\begin{split}
\langle \nu_{\pi(x)}^k, f(\pi(\xi))\rangle &:= \langle \nu_{x_{\pi(1)},...,x_{\pi(k)}}^k,f(\xi_{\pi(1)},...,\xi_{\pi(k)})\rangle\\
&= \langle \nu_x^k,f(\xi)\rangle
\end{split}
\end{equation*}
for a.e. $x \in (\T^2)^k$.
\item Consistency: Let $f\in C_0((\R^2)^k)$ and suppose that there exists some $g \in C_0((\R^2)^{k-1})$ such that $f(\xi_1,...,\xi_k) = g(\xi_1,...,\xi_{k-1})$ for all $\xi \in (\R^2)^k$. Then 
\[\langle \nu_{x_1,...,x_k}^k,f(\xi_1,...,\xi_k)\rangle = \langle \nu_{x_1,...,x_{k-1}}^{k-1},g(\xi_1,...,\xi_{k-1})\rangle \]
for a.e. $x \in (\T^2)^k$.
\item Diagonal continuity: $\lim_{r\to 0}\omega_r^2(\nu^2) = 0$, where we define the modulus of continuity
\[\omega_r^2(\nu^2) \coloneqq \int_{\T^2}\avint_{B_r(x)}\langle \nu_{x,y}^2,|\xi_1-\xi_2|^2\rangle\dd y\dx.\]  
\end{enumerate}
Each mapping $\nu^k$ is called a \emph{correlation marginal} and the set of all correlation measures will be denoted by $\CalL^2(\T^2;\R^2)$.
\end{definition}

The interpretation of a correlation measure $\corrM$ is to describe all multipoint-correlations of a flow, that is: Suppose $\euscr{F}$ is a set of velocity fields, $V_1,...,V_k \subset \R^2$ are measurable and $x=(x_1,...,x_k) \in (\T^2)^k$, then 
\[\nu^k_{x}(V_1\times ... \times V_k) = \text{Probability}\big(\lbrace u \in \euscr{F} : u(x_1) \in V_1, ...,u(x_k) \in V_k\rbrace\big).\]
This will be made more precise in \Cref{thm: corr measures main thm}.

For time-dependent correlation measures, we make the following adaptations as in \cite{FLMW20,FMW22}.

\begin{definition}
A {\em time-dependent correlation measure} is a collection $\corrM = (\nu^1,\nu^2,...)$ of maps $\nu^k \in \CalH_0^{k*}([0,T],\T^2;\R^2)$ satisfying the following properties:
\begin{enumerate}
\item $(\nu_t^1,\nu_t^2,...) \in \CalL^2(\T^2;\R^2)$ for a.e. $0 \leq t \leq T$.
\item $L^2$-boundedness: $\esssup_{0 \leq t \leq T}\int_{\T^2} \langle \nu_{t,x}^1,|\xi|^2\rangle \dx < \infty$.
\item Diagonal continuity: 
\begin{equation}\label{eq: diag cont}
\lim_{r\to 0}\int_0^T \omega_r^2(\nu_t^2)\dt = 0.
\end{equation} 
\end{enumerate}
The set of all such time-dependent correlation measures will be denoted by $\CalL^2([0,T],\T^2;\R^2)$.
\end{definition}

The simplest examples of correlation measures are the so-called \textit{atomic correlation measures}, where we consider a given function $u \in L^\infty(0,T;L^2(\T^2;\R^2))$ and define $\corrM = (\nu^1,\nu^2,...)$ as the hierarchy of product measures
\[\nu^k_{t,x_1,...,x_k} = \delta_{u(t,x_1)} \otimes ... \otimes \delta_{u(t,x_k)}.\]
In fact, in \cite{FLM17} it was argued that as a consequence of the diagonal continuity, given a Young measure $\lbrace \nu_x\rbrace_{x\in\T^2}$, the sequence of product measures $(\nu_x, \nu_x \otimes \nu_y,...)$ constitutes a correlation measure if and only if $\lbrace \nu_x \rbrace_{x\in\T^2}$ is a Dirac Young measure.

Analogously to \cite{FLM17} in the case of hyperbolic conservation laws, we now derive a notion of statistical solution of the Euler equations based on the evolution of correlation measures or, in other words, the evolution of moments. Evolving moments translates into evolving products in the state space $\R^2$. For this, we loosely use the notion of tensor product spaces $(\R^2)^{\otimes k}$ for every $k \in \N$. On $(\R^2)^{\otimes k}$, the product $:$ is defined by
\begin{equation}\label{eq: prod tensor prod space}
\eta : \zeta \coloneqq (\eta_1\cdot\zeta_1)...(\eta_k\cdot\zeta_k) \in \R
\end{equation}
for all $\eta = \eta_1\otimes ... \otimes \eta_k, \zeta = \zeta_1 \otimes ... \otimes \zeta_k \in (\R^2)^{\otimes k}$. This is well-defined due to the multilinearity of the right-hand side in \eqref{eq: prod tensor prod space}. Similarly, we define $:$ in $(\R^{2\times 2})^{\otimes k}$ and mixed tensor product spaces between $\R^2$ and $\R^{2 \times 2}$, where we then interpret the dot product $\cdot$ as the Frobenius inner product between matrices.\\
Now suppose $u\colon [0,T] \times \T^2 \to \R^2$ is a classical solution of the Euler equations with pressure $p\colon [0,T] \times \T^2 \to \R$ and initial datum $u(0) = u_0 \in C^1(\T^2;\R^2) \cap H$. We formally evolve the following tensor product using the product rule:
\begin{equation}\label{eq: evolution tensor prod}
\begin{split}
&\frac{d}{dt}(u(t,x_1) \otimes ... \otimes u(t,x_k))\\
=&\sum_{i=1}^k u(t,x_1) \otimes ... \otimes \partial_t u(t,x_i) \otimes ... \otimes u(t,x_k)\\
=&-\sum_{i=1}^k u(t,x_1) \otimes ... \otimes [\diverg(u(t,x_i)\otimes u(t,x_i)) + \nabla p(t,x_i)] \otimes ... \otimes u(t,x_k).
\end{split}
\end{equation} 

In order to interpret this equation in the sense of distributions, let $\theta \in C_c^1([0,T))$ and $g_1,...,g_k \in H^L$ and set 
\[g(x) = g_1(x_1) \otimes ... \otimes g_k(x_k), \quad x = (x_1,...,x_k) \in (\T^2)^k.\]
Then, from \eqref{eq: evolution tensor prod}, we further derive 
\begin{equation}\label{eq: evolution tensor prod weak}
\begin{split}
&-\int_0^T\int_{(\T^2)^k} \theta'(t)g(x) : (u(t,x_1) \otimes ... \otimes u(t,x_k))\\
&\qquad + \sum_{i=1}^k \theta(t)\nabla_{x_i}g(x) : (u(t,x_1) \otimes ... \otimes [u(t,x_i)\otimes u(t,x_i)] \otimes ... \otimes u(t,x_k))\dx\dt\\
&= \int_{(\T^2)^k} \theta(0)g(x) : (u_0(x_1)\otimes ... \otimes u_0(x_k))\dx,
\end{split}
\end{equation}

where $\nabla_{x_i}g(x) = g_1(x_1) \otimes ... \otimes \nabla_{x_i}g_i(x_i)\otimes ... \otimes g_k(x_k)$ for every $i = 1,...,k$.

Assuming that $u$ is in $L^\infty(0,T;H)$ and $u_0$ is in $H$, we now consider the atomic correlation measures $\corrM =  (\nu^1,\nu^2,...)$, $\overline{\corrM} = (\overline{\nu}^1,\overline{\nu}^2,...)$ given by
\[\nu_{t,x}^k = \delta_{u(t,x_1)} \otimes ... \otimes \delta_{u(t,x_k)},\quad \overline{\nu}_x^k = \delta_{u_0(x_1)} \otimes ... \otimes \delta_{u_0(x_k)}\]
for all $k\in\N$, $0 \leq t \leq T$ and $x = (x_1,...,x_k) \in (\T^2)^k$. Then we may rewrite \eqref{eq: evolution tensor prod weak} as 
\begin{equation}\label{eq: moment based Liouville}
\begin{split}
&-\int_0^T\int_{(\T^2)^k} \langle \nu^k_{t,x},\theta'(t)g(x) : (\xi_1 \otimes ... \otimes \xi_k)\rangle\\
&\qquad + \sum_{i=1}^k \langle \nu_{t,x}^k, \theta(t)\nabla_{x_i}g(x) : (\xi_1 \otimes ... \otimes [\xi_i\otimes \xi_i] \otimes ... \otimes \xi_k)\rangle\dx\dt\\
&= \int_{(\T^2)^k} \langle \overline{\nu}_{x}^k,\theta(0)g(x) : (\xi_1 \otimes ... \otimes \xi_k)\rangle\dx.
\end{split}
\end{equation}

In the following, we will usually write the factors which do not depend on $\xi$ in $\langle \nu_x,...\rangle$ outside of the brackets.

We now notice that the solution no longer explicitly appears in \eqref{eq: moment based Liouville} and we may use it as a defining equation. However, it comes at the cost of needing strong integrability conditions in order for \eqref{eq: moment based Liouville} to be well-defined in general. Following \cite{FMW22}, we therefore introduce two additional classes of functions, similar to $\CalH^k_0$, which we will also use later in \Cref{thm: corr meas comp thm}.

For every $k \in \N$, $\alpha \in \lbrace 0,1 \rbrace^k$ and $x \in (\T^2)^k$, we define $\overline{\alpha} \coloneqq (1-\alpha_1,...,1-\alpha_k)$, $|\alpha| \coloneqq \alpha_1 + ... + \alpha_k$ and $x_{\alpha} \in (\T^2)^{|\alpha|}$ consists of those entries $x_i$ of $x$ for which $\alpha_i = 1$. Moreover, for every $i = 1,...,k$, we define $\hat{x}^i \coloneqq (x_1,...,x_{i-1},x_{i+1},...,x_k) \in (\T^2)^{k-1}$.\\
Then, for every $k \in \N$, the space $\CalH^{k,2}([0,T],\T^2;\R^2)$ consists of all measurable functions $f\colon [0,T]\times(\T^2)^k \times (\R^2)^k \to \R$ such that $\xi \mapsto f(t,x,\xi)$ is continuous for a.e. $(t,x) \in [0,T] \times (\T^2)^k$ and
\[|f(t,x,\xi)| \leq \sum_{\alpha \in \lbrace 0,1 \rbrace^k} \varphi_{|\overline{a}|}(t,x_{\overline{\alpha}}) |\xi^\alpha|^2,\quad (t,x) \in [0,T]\times(\T^2)^k, \quad\xi \in (\R^2)^k\]
for non-negative $\varphi_i \in L^\infty(0,T;L^1((\T^2)^i)), i = 0,1,...,k$ (with the agreement that $L^1((\T^2)^0) = \R$).

Then $\CalH^{k,2}_1([0,T],\T^2;\R^2)$ is the subspace of functions $f \in \CalH^{k,2}([0,T],\T^2;\R^2)$ which satisfy a Lipschitz condition in the sense that for some $r > 0$ and non-negative $h \in \CalH^{k-1,2}([0,T],\T^2;\R^2), \tilde{h} \in \CalH^{k,2}([0,T],\T^2;\R^2), \psi\in L^\infty(0,T)$, we have
\begin{equation}
|f(t,x,\xi) - f(t,x,\zeta)| \leq \sum_{i=1}^k\psi(t)|\xi_i - \zeta_i|\max\lbrace |\xi_i|,|\zeta_i|\rbrace h(t,\hat{x}^i, \hat{\xi}^i) + \CalO(|x-y|)\tilde{h}(t,x,\xi)
\end{equation}
for all $x \in (\T^2)^k, y \in B_r(x), \xi,\zeta \in (\R^2)^k$. In the time-independent case, $\CalH^{k,2}(\T^2;\R^2)$ is defined analogously in an obvious manner.
Aside from these definitions, we will not further discuss these function spaces. We note that every integrand in \eqref{eq: moment based Liouville} lies in $\CalH^{k,2}([0,T],\T^2;\R^2)$ (recall that $H^L \hookrightarrow C^{1,\beta}(\T^2;\R^2)$) and we will say that a correlation measure $\corrM$ has integrable $\CalH^{\cdot,2}$ moments if for every $k \in \N$ and $f \in \CalH^{k,2}([0,T],\T^2;\R^2)$, $\langle\nu_{t,x}^k, f\rangle$ is integrable over $[0,T]\times(\T^2)^k$ (or simply $(\T^2)^k$ in the time-independent case).

\begin{definition}\label{def: stats sol moment based}
Let $\overline{\corrM} \in \CalL^2(\T^2;\R^2)$ be a correlation measure having integrable $\CalH^{\cdot,2}$ moments. A time-dependent correlation measure $\corrM \in \CalL^2([0,T],\T^2;\R^2)$ having integrable $\CalH^{\cdot,2}$ moments is called a \emph{moment based statistical solution} of the (2D incompressible) Euler equations with initial correlation measure $\overline{\corrM}$ if 
\begin{enumerate}[i)]
\item for all $\psi \in C^\infty(\T^2)$, we have
\begin{equation}\label{eq: corr meas div free}
\int_{(\T^2)^2}\langle \nu^2_{t,x_1,x_2},\xi_1 \otimes \xi_2\rangle : (\nabla \psi(x_1) \otimes \nabla\psi(x_2))\dx = 0
\end{equation}
for a.e. $0 \leq t \leq T$,
\item the energy inequality
\begin{equation}\label{eq: moment based energy inequ}
\int_{\T^2} \langle \nu_{t,x}^1,|\xi|^2\rangle \dx \leq \int_{\T^2} \langle \overline{\nu}_{x}^1, |\xi|^2\rangle \dx
\end{equation}
holds for a.e. $0 \leq t \leq T$,
\item for every $k \in \N$, equation \eqref{eq: moment based Liouville} is satisfied by $\corrM$ for all $\theta \in C_c^1([0,T))$ and $g(x) = g_1(x_1) \otimes ... \otimes g_k(x_k), x = (x_1,...,x_k) \in (\T^2)^k,$
where $g_1,...,g_k \in H^L$.
\end{enumerate}
\end{definition}

The integrand in \eqref{eq: corr meas div free} can be thought of as the square of the divergence. Property i) in the above definition is therefore a weak divergence-free condition as will be made precise in \Cref{lem: equiv div free}.\\
Given a function $u \in L^\infty(0,T;H)$ and $u_0 \in H$, it is not hard to see that $u$ is a weak solution of the Euler equations with initial datum $u_0$ if and only if the corresponding atomic correlation measure $(\delta_{u(t,x)}, \delta_{u(t,x)} \otimes \delta_{u(t,y)},...)$ is a moment-based statistical solution with initial correlation measure $(\delta_{u_0(x)}, \delta_{u_0(x)} \otimes \delta_{u_0(y)},...)$. This holds since for $k = 1$, \eqref{eq: moment based Liouville} is just the weak formulation of the Euler equations, \eqref{eq: moment based energy inequ} is equivalent to the energy inequality and we note that any weak solution $u \in L^\infty(0,T;H)$ of the Euler equations is already in $C([0,T];H_w)$ (or has a representative therein to be precise).\\
We now move on to show an equivalence between phase space statistical solutions as in \Cref{def: stats sol phase space} and the just introduced moment based statistical solutions in \Cref{def: stats sol moment based}. In the context of the Navier-Stokes equations, this was shown in \cite{FMW22}. 
The essential tool is the main theorem for correlation measures \cite[Theorem 2.7]{FLM17}, of which we will state the more general time-dependent version here \cite[Theorem 2.20]{FLMW20}.

\begin{theorem}\label{thm: corr measures main thm}
Let $\corrM \in \CalL^2([0,T],\T^2;\R^2)$ be a time-dependent correlation measure. Up to a Lebesgue null set, there exists a family $\lbrace \mu_t\rbrace_{0 \leq t \leq T}$ of Borel probability measures on $L^2(\T^2;\R^2)$ such that
\begin{enumerate}[i)]
\item the map
\begin{equation}\label{eq: meas main thm corr meas}
t \mapsto \int_{L^2}\int_{(\T^2)^k} h(x,u(x))\dx\dd\mu_t(u)
\end{equation}
is measurable for all $h \in \CalH_0^k(\T^2;\R^2)$,
\item the second moments w.r.t. $\|\cdot\|_{L^2}$ are uniformly bounded in time, i.e.
\begin{equation}\label{eq: finite sec mom}
\esssup_{0 \leq t \leq T} \int_{L^2} \|u\|_{L^2}^2\dd\mu_t(u) < \infty,
\end{equation}
\item the identity
\begin{equation}\label{eq: corr meas main thm identity}
\int_{(\T^2)^k} \langle \nu_t^k,h(x,\xi)\rangle\dx = \int_{L^2}\int_{(\T^2)^k}h(x,u(x))\dx\dd\mu_t(u)
\end{equation}
holds for a.e. $t \in [0,T]$, every $h \in \CalH^k_0(\T^2;\R^2)$, and all $k \in \N$.
\end{enumerate}
Conversely, for every family $\lbrace \mu_t \rbrace_{0 \leq t \leq T}$ satisfying i) and ii), there is a unique correlation measure $\corrM \in \CalL^2([0,T],\T^2;\R^2)$ satisfying iii).
\end{theorem}

We may then formulate the equivalence between moment based and phase-space statistical solutions of the Euler equations as follows.

\begin{theorem}\label{thm: main thm moment based <-> phase space}
Let $\overline{\mu}, \lbrace \mu_t \rbrace_{0\leq t \leq T}$ be Borel probability measures on $H$ and suppose that there exists $R > 0$ such that
\begin{equation}\label{eq: main thm bounded supp}
\supp\overline{\mu}, \supp\mu_t \subset B_R^H
\end{equation}
for a.e. $t \in [0,T]$. Also, let $\lbrace \mu_t \rbrace_{0 \leq t \leq T}$ satisfy \eqref{eq: finite sec mom} and one of the equivalent measurability conditions \eqref{eq: meas main thm corr meas} or \eqref{eq: phase space stats sol measurability} (see \Cref{prop: equiv meas}).\\ 
If we consider the associated correlation measures $\overline{\corrM} \in \CalL^2(\T^2;\R^2)$ and $\corrM \in \CalL^2([0,T],\T^2;\R^2)$, then $\corrM$ is a moment based statistical solution of the Euler equations with initial correlation measure $\overline{\corrM}$ if and only if (after redefining on a Lebesgue null set) $\lbrace \mu_t \rbrace_{0\leq t \leq T}$ is a phase space statistical solution of the Euler equations with initial distribution $\overline{\mu}$.  
\end{theorem}

\begin{remark}
Phase space statistical solutions of the Euler equations which come from trajectory statistical solutions with initial distribution of bounded support in $H$ have uniformly bounded support in $H$ as a consequence of the energy inequality. Generally, this is unknown and may not need to be the case.
\end{remark}

It can be proved that a Borel probability measure on $L^2(\T^2;\R^2)$ being concentrated on the closed, hence measurable, subset $H$ of $L^2(\T^2;\R^2)$ is equivalent to the corresponding correlation measure satisfying \eqref{eq: corr meas div free}.

\begin{lemma}\label{lem: equiv div free}(\cite[Lemma 3.1]{LMP21})
Let $\mu$ be a Borel probability measure on $L^2(\T^2;\R^2)$, corresponding to a correlation measure $\corrM \in \CalL^2(\T^2;\R^2)$. Then $\mu$ is concentrated on $H$ if and only if \eqref{eq: corr meas div free} holds.
\end{lemma}

Also, the measurability conditions i) in Theorem~\ref{thm: corr measures main thm} and i) in \Cref{def: stats sol phase space} are equivalent. 

\begin{lemma}\label{prop: equiv meas}
Let $\lbrace \mu_t \rbrace_{0 \leq t \leq T}$ be a family of Borel probability measures on $H$. Then the following are equivalent:
\begin{enumerate}[i)]
\item The map
\begin{equation}\label{eq: equiv meas 1}
t \mapsto \int_H\int_{(\T^2)^k} h(x,u(x))\dx\dd\mu_t(u)
\end{equation}
is measurable for all $h \in \CalH^k_0(\T^2;\R^2)$ and all $k \in \N$.
\item The map
\begin{equation}\label{eq: equiv meas 2}
t \mapsto \int_H \varphi(u)\dd\mu_t(u)
\end{equation}
is measurable for all $\varphi \in C_b(H_w)$. 
\end{enumerate}
\end{lemma}

Very similarly to \cite{FMW22}, we employ a monotone class argument (see for instance \cite[Theorem A.1]{J97}).
 
\begin{theorem}\label{thm: monotone class}
Let $W$ be a vector space of bounded real functions on a set $\Omega$ that contains the constant functions and is closed under bounded monotone convergence. If $M$ is a subset of $W$ which is closed under multiplication, then $W$ contains every bounded function that is measurable with respect to the $\sigma$-algebra on $\Omega$ generated by $M$.
\end{theorem}

We are also going to use the following result from \cite[Lemma 2.10]{FLM17}. 

\begin{lemma}\label{lem: caratheodory funct cont}
For every $h \in \CalH^k_0(\T^2;\R^2)$ and $k \in \N$, the function
\begin{equation}\label{eq: integral caratheodory funct}
u \mapsto \int_{(\T^2)^k} h(x,u(x))\dx
\end{equation}
is continuous on $H$.
\end{lemma}

In the following proof of \Cref{prop: equiv meas}, we let $(\CalJ_\varepsilon)_{\varepsilon > 0}$ be the family of smoothing operators associated to a family of standard mollifiers.

\begin{proof}[Proof of \Cref{prop: equiv meas}]
We first assume ii). Then, let $k \in \N, h \in \CalH^k_0(\T^2;\R^2)$ and define
\[\varphi(u) \coloneqq \int_{(\T^2)^k} h(x,u(x))\dx, \quad u \in H\]
as well as 
\[\varphi_{\varepsilon}(u) \coloneqq \varphi(\CalJ_\varepsilon u) = \int_{(\T^2)^k} h(x,\CalJ_\varepsilon u(x))\dx, \quad u \in H,\]
for every $\varepsilon > 0$. We note that whenever $u^n \rightharpoonup u\,(n \to \infty)$ in $H$, then $\CalJ_\varepsilon u^n \to \CalJ_\varepsilon u\,(n \to \infty)$ pointwise on $\T^2$. Since $h$ is continuous in the second argument, the dominated convergence theorem yields $\varphi_{\varepsilon} \in C_b(H_w)$. In particular,
\[t \mapsto \int_H \varphi_{\varepsilon}(u)\dd\mu_t(u)\]
is measurable.\\
Due to the continuity of $\varphi$ with respect to $\|\cdot\|_{L^2}$ by \Cref{lem: caratheodory funct cont} and the convergence $\CalJ_{\varepsilon}u \to u\,(\varepsilon \to 0)$ in $H$ for any $u \in H$, we obtain pointwise convergence $\varphi_{\varepsilon} \to \varphi\,(\varepsilon \to 0)$ on $H$. Then the dominated convergence theorem yields measurability of $t \mapsto \int_H\varphi(u)\dd\mu_t(u)$ as the a.e.\ pointwise limit of measurable functions.

Now we assume i) and show the converse. To apply \Cref{thm: monotone class}, we consider $\Omega = H$, $W$ to be the set of bounded and Borel measurable functions $\Psi\colon H \to \R$ such that $t \mapsto \int_H \Psi(u)\dd\mu_t(u)$ is measurable, and $M$ the space of all functions as in \eqref{eq: integral caratheodory funct} for $h \in \CalH^k_0(\T^2;\R^2)$ and all $k \in \N$. Our assumption i) precisely guarantees that $M \subset W$.

It can be shown that with these definitions, the assumptions of \Cref{thm: monotone class} are satisfied. We denote the $\sigma$-algebra generated by $M$ with $\CalA$.

For any fixed $u_0 \in H$, by considering smooth cut-off functions $\chi \in C_c^\infty(\R^2)$, one can use the approximations $(x,\xi) \mapsto |\xi - u_0(x)|\chi(\xi) \in \CalH^1_0(\T^2;\R^2)$ to show that $u \mapsto \int_{\T^2}|u(x)-u_0(x)|^2\dx = \|u-u_0\|_{L^2}^2$ is $\CalA$-measurable and consequently, all open balls in $H$ are $\CalA$-measurable. In particular, $\CalA$ contains the Borel-$\sigma$-algebra on $H$. The latter is also generated by the $H_w$ topology as remarked before so that any $\varphi \in C_b(H_w)$ belongs to $W$ as desired.
\end{proof}

Before finally proving \Cref{thm: main thm moment based <-> phase space}, we give an equivalent formulation of the Foia\cs-Liouville equation, which can be compared more easily to \eqref{eq: moment based Liouville}. The alternative formulation \eqref{eq: Foias-Liouville equiv form} to \eqref{eq: foias-liouville def} in \Cref{def: stats sol phase space} almost everywhere can be proved as in \cite[§3 Lemma 5]{F72}. Altering the family of measures on a Lebesgue null set so that \eqref{eq: foias-liouville def} holds everywhere can be done as in \cite[§3 Theorem 2]{F72} under the assumption of uniformly bounded supports, which we also make in \Cref{thm: main thm moment based <-> phase space}.

\begin{lemma}\label{lem: equiv foias-liouville}
Let $\overline{\mu}$ and $\lbrace \mu_t \rbrace_{0\leq t \leq T}$ be Borel probability measures on $H$ satisfying i) and ii) in \Cref{def: stats sol phase space} and suppose that there exists $R > 0$ such that for a.e. $0 \leq t \leq T$
\[\supp\overline{\mu},\supp\mu_t \subset B_R^H.\]
Then iii) in \Cref{def: stats sol phase space} holds true if (after altering $\lbrace \mu_t\rbrace_{0\leq t \leq T}$ on a Lebesgue null set) and only if
\begin{equation}\label{eq: Foias-Liouville equiv form}
-\int_0^T\int_H \theta'(t)\Phi(u) + (u\otimes u,\theta(t)\nabla\Phi(u))_{L^2}\dd\mu_t(u)\dt = \int_H\theta(0)\Phi(u_0)\dd\overline{\mu}(u_0)
\end{equation}
for all $\Phi \in \CalT$ and $\theta \in C^1_c([0,T))$. 
\end{lemma}

\begin{proof}[Proof of \Cref{thm: main thm moment based <-> phase space}]
Due to \Cref{lem: equiv div free} and \Cref{prop: equiv meas}, all that remains to show is the equivalence between the versions of the Foia\cs-Liouville equations \eqref{eq: Foias-Liouville equiv form} and \eqref{eq: moment based Liouville} as well as the equivalence of the energy inequalities \eqref{eq: phase space energy inequ} and \eqref{eq: moment based energy inequ}. The latter, however, follows in a rather straightforward way from \Cref{thm: corr measures main thm}.\\
Suppose first that $\lbrace \mu_t \rbrace_{0\leq t \leq T}$ is a phase-space statistical solution with initial distribution $\overline{\mu}$ and uniformly bounded support in $B_R^H$. We then show that $\corrM, \overline{\corrM}$ satisfy \eqref{eq: moment based Liouville}.\\
Let $k \in \N, g_1,...,g_k \in H^L$ as well as $\theta \in C_c^1([0,T))$ and consider $g = g_1 \otimes ... \otimes g_k$ as in \Cref{def: stats sol moment based} iii). For a smooth cut-off function $\chi \in C_c^\infty(\R^k)$, we define
\[p(s) = s_1...s_k\chi(s),\quad s \in \R^k,\]
so that $p \in C^1_c(\R^k)$. Using $\Phi(u) = p((u,g)_{L^2}^k)$, equation \eqref{eq: Foias-Liouville equiv form} reads as
\begin{equation}\label{eq: main thm proof eq1}
\begin{split}
&-\int_0^T\int_H \theta'(t)\prod_{j=1}^k (u,g_j)_{L^2}\chi((u,g)_{L^2}^k)\\
&\qquad+\bigg( u \otimes u, \theta(t)\sum_{i=1}^k \prod_{\substack{j=1\\j\neq i}}^k (u,g_j)_{L^2}\nabla g_i\chi((u,g)_{L^2}^k) \bigg)_{L^2}\dd\mu_t(u)\dt + \varepsilon(\chi)\\
&= \int_H \theta(0)\prod_{j=1}^k (u_0,g_j)_{L^2}\chi((u_0,g)_{L^2}^k)\dd\overline{\mu}(u_0)\\
\Leftrightarrow 
&-\int_0^T\int_H \theta'(t)\prod_{j=1}^k (u,g_j)_{L^2}\chi((u,g)_{L^2}^k)\\
&+\theta(t)\sum_{i=1}^k(u,g_1)_{L^2} ... (u,g_{i-1})_{L^2}(u \otimes u, \nabla g_i)_{L^2}(u,g_{i+1})_{L^2} ... (u,g_k)_{L^2}\chi((u,g)_{L^2}^k)\dd\mu_t(u)\dt 
 + \varepsilon(\chi)\\
&= \int_H \theta(0)\prod_{j=1}^k (u_0,g_j)_{L^2}\chi((u_0,g)_{L^2}^k)\dd\overline{\mu}(u_0),
\end{split}
\end{equation}
where $\varepsilon(\chi) = -\int_0^T\int_H (u \otimes u, \theta(t)\prod_{j=1}^k (u,g_j)_{L^2}\sum_{i=1}^k\partial_i\chi((u,g)_{L^2}^k)g_i)_{L^2}\dd\mu_t(u)\dt$ is the error term depending on $\chi$. However, choosing a sequence of smooth cut-off functions $(\chi_n)_{n\in\N} \in C_c^\infty(\R^k)$ such that 
\[0 \leq \chi_n \leq 1, \chi_n \equiv 1 \text{ on } B_n^H \text{ and } \|\nabla \chi_n\|_{L^\infty} \leq 1,\]
the uniform bound of the supports \eqref{eq: main thm bounded supp} and the dominated convergence theorem yield that \eqref{eq: main thm proof eq1} also holds in the limit, i.e. without $\chi$ and $\varepsilon(\chi)$. We particularly used that each integrand in \eqref{eq: main thm proof eq1} and the error term $\varepsilon$ can be bounded by a constant times $\|\theta\|_{C^1}\max_{i=1,...,k}\|\g_i\|_{C^1}(1+\|u\|_{L^2}^{k+2})$. Using this estimate along with the uniform bound \eqref{eq: main thm bounded supp} of the supports $\lbrace \mu_t \rbrace_{0 \leq t \leq T}$ and $\overline{\mu}$, one can also show by a standard cut-off argument that the identity in \eqref{eq: corr meas main thm identity} may be applied to the above integrals. Transforming \eqref{eq: main thm proof eq1} immediately yields \eqref{eq: moment based Liouville}.\\
We now suppose that $\corrM$ is a moment based statistical solution with initial correlation measure $\overline{\corrM}$ and show the converse, i.e. that $\lbrace \mu_t \rbrace_{0 \leq t \leq T}$ is a phase space statistical solution with initial distribution $\overline{\mu}$.\\
Let $\phi \in C_c^1(\R^k)$, $g_1,...,g_k \in H^L$ and $\theta \in C^1_c([0,T))$.\\
There exists a sequence of multivariate polynomials $(P_n)_{n\in\N}$ such that $P_n \to \phi\,(n\to\infty)$ in $C^1([-l,l]^k)$ (see for instance \cite{VV16}), where we will choose $l > 0$ large enough later. Fix $n \in \N$ for the moment and write $P_n$ in the form
\[P_n(s) = \sum_{|\alpha|\leq N}a_\alpha s_1^{\alpha_1}...s_k^{\alpha_k},\]
where $|\alpha|$ denotes the order of a multi-index $\alpha = (\alpha_1,...,\alpha_k) \in \N^k$ and $a_\alpha \in \R$ are real coefficients for each such multi-index $\alpha$.\\
Then we may evolve the moments $(|\alpha| \leq N)$
\[(\xi_1 \otimes ... \otimes \xi_{\alpha_1}) \otimes ... \otimes (\xi_{|\alpha|-\alpha_k} \otimes ... \otimes \xi_{|\alpha|})\]
tested against 
\[(g_1(x_1) \otimes ... \otimes g_1(x_{\alpha_1})) \otimes ... \otimes (g_k(x_{|\alpha|-\alpha_k}) \otimes ... \otimes g_k(x_{|\alpha|}))\]
according to \eqref{eq: moment based Liouville}. We may also consider the weighted sums $\sum_{|\alpha| \leq N}$ of these equations.\\
The first resulting term on the left-hand side in \eqref{eq: moment based Liouville} 
\begin{align*}
&\int_0^T \int_{(\T^2)^k} \langle \nu_{t,x}^{|\alpha|}, \theta'(t)(\xi_1 \otimes ... \otimes \xi_{\alpha_1}) \otimes ... \otimes (\xi_{|\alpha|-\alpha_k} \otimes ... \otimes \xi_{|\alpha|})\rangle :\\
&\qquad : (g_1(x_1) \otimes ... \otimes g_1(x_{\alpha_1})) \otimes ... \otimes (g_k(x_{|\alpha|-\alpha_k}) \otimes ... \otimes g_k(x_{|\alpha|}))\dx\dt
\end{align*}
transforms via \eqref{eq: corr meas main thm identity} to
\[\int_0^T \int_H \theta'(t)(u,g_1)_{L^2}^{\alpha_1}...(u,g_k)_{L^2}^{\alpha_k}\dd\mu_t(u)\dt.\]
We will not write out the other terms as this gets quite messy and rather note that all terms can be similarly transformed so that multiplied by $a_\alpha$ and summed up over $|\alpha| \leq N$ we obtain
\begin{equation}
\begin{split}
&-\int_0^T\int_H \theta'(t)P_n((u,g)_{L^2}^k) + \sum_{i=1}^k(u \otimes u, \theta(t)\partial_i P_n((u,g)_{L^2}^k)\nabla g_i)_{L^2}\dd\mu_t(u)\dt\\
&\qquad =\int_0^T\int_H \theta(0)P_n((u,g)_{L^2}^k)\dd\overline{\mu}(u)\dt.
\end{split}
\end{equation}

The arguments of $P_n$ and $\partial_i P_n$ in the integral are bounded in the Euclidean distance from above due to
\begin{equation}\label{eq: bound (u,g)}
|(u,g)_{L^2}^k| \leq \sqrt{k}\max_{i=1,...,k}\|g_i\|_{L^2}\|u\|_{L^2}^k.
\end{equation}
As $\mu_t$, for a.e. $0 \leq t \leq T$, and $\overline{\mu}$ have support in $B_R^H$, choosing $l = \sqrt{k}\max_{i=1,...,k}\|g_i\|_{L^2}R^k$ above and making use of the uniform approximation of $\phi$ by $(P_n)_{n\in\N}$ in $C^1([-l,l]^k)$, passing to the limit $(n \to \infty)$ yields \eqref{eq: Foias-Liouville equiv form}.
\end{proof}

\begin{remark}
In the proof above of \Cref{thm: main thm moment based <-> phase space}, it seems difficult to significantly relax the assumptions of uniformly bounded supports of $\overline{\mu}$ and $\lbrace \mu_t \rbrace_{0\leq t \leq T}$ due to the uniform approximation of a compactly supported function by polynomials. 
\end{remark}

\subsection{Construction of Statistical Solutions}\label{subsec: construction stats sols}
In this subsection, we would like to comment on and draw some comparisons between established methods by which one can construct the statistical solutions that we described in the previous part. We will focus here on the following three methods:
\begin{enumerate}[I)]
\item Compactness of approximating sequences of time-parametrized measures;
\item Discrete approximations based on the Krein-Milman theorem;
\item Measurable selections and push-forward constructions. 
\end{enumerate}

However, we would at least like to mention that for (non-autonomous) dissipative dynamical systems, there is a fairly general way of using the theory of (pullback) attractors and Banach limits to construct statistical and stationary statistical solutions. We refer the reader to some of the work by Foia\cs, Rosa and Teman \cite{FRT15,FRT19}, by \L{}ukaszewicz et al. \cite{L08,LRR11,LR14}, see also \cite{LK16}, and Zhao et al. \cite{CLZ20,LSZ20,CSZ20}.

For each of the aforementioned methods, we will sketch how they can be applied to prove the following existence result.

\begin{theorem}\label{thm: general existence thm stats sols}
Let $\overline{\mu}$ be a Borel probability measure on $H$ satisfying
\begin{equation}\label{eq: general ex thm, bdd mean enstrophy}
\int_{H}\|u_0\|_{H^1}^2\dd\overline{\mu}(u_0) < \infty.
\end{equation}
Then there exists a phase space statistical solution of the 2D incompressible Euler equations $\lbrace \mu_t \rbrace_{0\leq t \leq T}$ with initial distribution $\overline{\mu}$ and the energy inequality
\begin{equation}\label{eq: energy inequ construction}
\int_H \|u\|_{L^2}^2\dd\mu_t(u) \leq \int_H \|u_0\|^2_{L^2}\dd\overline{\mu}(u_0)
\end{equation}
holds for a.e.\ every $0 \leq t \leq T$. 
\end{theorem}


Under the assumption of initial data in $H^1$, there exist weak solutions of the Euler equations continuous with values in $H$ or even $H^1$, not just in $H_w$. In the constructions in II) and III), we may therefore more naturally consider statistical solutions as Borel probability measures on $H$ or $C([0,T];H)$, in contrast to the slightly more general formulations of the previous section.

We formulate the following statement globally in time in order to also be able to apply it in III), when considering semiflow selections.

\begin{prop}\label{prop: existence euler}
Let $u_0 \in H^1$, then there exists a global weak solution $u \in C_{loc}([0,\infty);H^1)$ of the Euler equations such that its vorticity $\omega(u) \in C_{loc}([0,\infty);L^2(\T^2))$ is a renormalized solution of the vorticity formulation of the Euler equations and the following properties hold:
\begin{enumerate}[i)]
\item $\|u(t)\|_{L^2} \leq \|u(t')\|_{L^2}$ for all $0 \leq t' \leq t < \infty$,
\item $\|\omega(u)(t)\|_{L^2} = \|\omega(u)(0)\|_{L^2}$ for all $0 \leq t < \infty$,
\item $\esssup_{t' \leq \tau \leq t} \|\partial_t u(\tau)\|_{H^{-s}} \leq C\|u(t')\|_{L^2}$ for some $s > 0, C = C(s) > 0$ and all $0 \leq t' \leq t < \infty$. 
\end{enumerate}
\end{prop} 

Here, the {\em vorticity} of a flow field $u$ is defined as $\omega=\partial_{x_1}u_2-\partial_{x_2}u_1$. The {\em vorticity formulation} of the Euler equations is obtained by taking the curl of the momentum equation, whereby 
\begin{equation}\label{eq: vort equ}
\partial_t\omega+u\cdot\nabla\omega =0.
\end{equation}
A distributional solution $\omega$ of this vorticity transport equation is called {\em renormalized} if for every $\beta\in C^1(\R;\R)$ with $\beta(0)=0$ and bounded derivative, also
\begin{equation*}
\partial_t\beta(\omega)+u\cdot\nabla\beta(\omega) =0.
\end{equation*}

The subscript $loc$ here means that we consider the compact-open topology on these spaces of continuous functions. The proof of this proposition on arbitrary finite time intervals is somewhat classical (cf. \cite{DM87conc,DL89,LMN05,HLNS99}). For global existence, one can employ a diagonal argument.
\hfill\\

\subsubsection{Compactness of Approximating Sequences of Time-Parametrized Measures}
The idea here is to construct a statistical solution $\lbrace \mu_t \rbrace_{0 \leq t \leq T}$ by considering approximating sequences of families of Borel probability measures $(\lbrace \mu_t^N \rbrace_{0 \leq t \leq T})_{N\in\N}$ and using compactness theorems.

We state here prototypically the arguments from Foia\cs' first article on statistical solutions of the Navier-Stokes equations \cite{F72}. There, given an appropriate fixed orthonormal base of $H$, the $N$-th order Galerkin solution operators $S_t^N$,$0\leq t \leq T$, were considered. Then, given a Borel probability measure $\overline{\mu}$ on $H$, the approximations $\mu_t^N \coloneqq {S_t^N}_\sharp\overline{\mu}$ were studied.

We summarize the compactness arguments that Foia\cs\, derived in the following theorem, which hold generally as stated and not just for the above example of pushforward Galerkin approximations. For this, we define $C_2(H)$ to be the space of all continuous functions $\varphi: H \to \R$ such that $\|\varphi\|_{C_2(H)} \coloneqq \sup_{u \in H} \frac{|\varphi(u)|}{1+\|u\|^2_{L^2}} < \infty$. Then $(C_2(H),\|\cdot\|_{C_2(H)})$ is a Banach space, not reflexive nor separable, which was a main difficulty in the proof.

Also, we recall that for a topological space $X$, we say that for a subset $A$ of $X$, $x \in X$ is an accumulation point if for every open neighborhood $U$ of $x$, $A \cap (U \setminus \lbrace x \rbrace)$ is nonempty.\\
As usual and previously used, the weak-* topology on the dual of a Banach space $Y$ is the coarsest topology for which the linear functionals $\lbrace f_y \rbrace_{y \in Y}$ on $Y'$, given by $f_y(\varphi) = \varphi(y)$ for all $\varphi \in Y',y \in Y$, are continuous.

In particular, if $\varphi \in Y'$ is a weak-* accumulation point of $A \subset Y'$, then for every $\varepsilon > 0$ and $y_1,...,y_k \in Y$, $A \cap \lbrace \psi \in Y' : |\psi(y_i) - \varphi(y_i)| < \varepsilon \text{ for all } i = 1,...,k\rbrace \neq \emptyset$.

The following theorem holds in both two and three dimensions.

\begin{theorem}\label{thm: foias comp thm}
Suppose that a sequence $(\lbrace \mu_t^{N} \rbrace_{0 \leq t \leq T})_{N\in\N}$ of families of Borel probability measures on $H$ satisfies
\begin{equation}\label{eq: foias comp thm assump1}
t \mapsto \int_H \varphi(u)\dd\mu_t^N(u) \text{ is Lebesgue measurable for every } N\in\N, \varphi \in C_b(H),
\end{equation}
\begin{equation}\label{eq: foias comp thm assump2}
C_1 \coloneqq \sup_{N\in\N}\esssup_{0\leq t \leq T} \int_H \|u\|_{L^2}^2\dd\mu_t^{N}(u) < \infty,
\end{equation}
\begin{equation}\label{eq: foias comp thm assump3}
C_2 \coloneqq \sup_{N\in\N}\int_0^T \int_H \|u\|_{H^1}^2\dd\mu_t^{N}(u)\dt < \infty,
\end{equation}
and consider the corresponding functionals $(F^{N})_{N\in\N} \subset L^1(0,T;C_2(H))'$, given by $F^{N}(\Phi) = \int_0^T\int_H \Phi(t,u)\dd\mu_t^{N}(u)\dt$ for all $N \in \N$ and $\Phi \in L^1(0,T;C_2(H))$. Then the following holds:
\begin{enumerate}[i)]
\item There exists a weak-* accumulation point $F$ of $(F^{N})_{N\in\N}$ in $L^1(0,T;C_2(H))'$ and a function $G: (0,T) \to C_2(H)'$ that represents $F$ in the sense that 
\begin{equation*}
\sup_{0<t<T}\|G(t)\|_{C_2(H)'} = \|F\|_{L^1(0,T;C_2(H))'}
\end{equation*}
 and
\begin{equation*}
F(\Phi) = \int_0^T \langle G(t),\Phi(t) \rangle\dt
\end{equation*}
for all $\Phi \in L^1(0,T;C_2(H))$, where $\langle G(t),\Phi(t)\rangle := G(t)(\Phi(t))$. Furthermore, there exists a family of probability measures $\lbrace \mu_t\rbrace_{0 \leq t \leq T}$ also satisfying \eqref{eq: foias comp thm assump1}-\eqref{eq: foias comp thm assump3} and a Lebesgue null set $E \subset (0,T)$ such that
\begin{equation}
\langle G(t),\varphi \rangle = \int_H \varphi(u)\dd\mu_t(u) \text{ and } F(\Phi) = \int_0^T \int_H \Phi(t,u)\dd\mu_t(u)\dt
\end{equation}
for all $t\in (0,T)\setminus E$ and all $\varphi \in C_b(H), \Phi \in L^1(0,T;C_b(H))$.
\item Suppose that additionally $\|\cdot\|_{L^2}^2$ is uniformly integrable with respect to the sequence $(\lbrace \mu_t^{(N)}\rbrace_{0\leq t\leq T} \rbrace)_{N\in\N}$ in the sense that
\begin{equation}\label{eq: foias comp thm assump4}
\lim_{r\to \infty}\sup_{N\in\N} \esssup_{0\leq t \leq T}\int_{\lbrace \|u\|_{L^2}>r \rbrace} \|u\|_{L^2}^2\dd\mu^{(N)}_t(u) = 0.
\end{equation}
Then
\[F(\Phi) = \int_0^T\int_H \Phi(t,u)\dd\mu_t(u)\dt \]
for all $\Phi \in L^1(0,T;C_2(H))$.
\end{enumerate}
\end{theorem}

Under an assumption of the form \eqref{eq: general ex thm, bdd mean enstrophy} on the mean initial vorticity, Chae in \cite{C91p} essentially argued along the same lines to obtain a first existence result for statistical solutions of the 2D incompressible Euler equations on a periodic domain, though the issue of $C_2(H)$ not being separable was kind of ignored. We sketch the proof here. In before, we need to point out that in the following, from \Cref{thm: foias comp thm}, we may only obtain a family of measures $\lbrace \mu_t \rbrace_{0\leq t \leq T}$, representing an accumulation point, which satisfies \eqref{eq: Foias-Liouville equiv form} instead of \eqref{eq: foias-liouville def}. We remark that we could apply \Cref{lem: equiv foias-liouville} in the following proof to actually obtain \eqref{eq: foias-liouville def} if $\overline{\mu}$ had bounded support in $H$ (the bounded support of $\lbrace \mu_t \rbrace_{0\leq t \leq T}$ would then follow from the construction).

\begin{proof}[Proof of \Cref{thm: general existence thm stats sols} (Sketch 1)]
We consider vanishing viscosity approximations. Therefore, let $\varepsilon > 0$ denote the kinematic viscosity and consider statistical solutions $\lbrace \mu_t^\varepsilon \rbrace_{0 \leq t \leq T}$ of the 2D Navier-Stokes equations with viscosity $\varepsilon$ and initial distribution $\overline{\mu}$ (cf. \cite[Theorems V.1.1, V.1.2]{FMRT01}). As the Navier-Stokes equations yield unique weak solutions, this statistical solution can be constructed as the pushforward measure of the initial distribution along a solution operator. Consequently, energy and enstrophy inequalities, among other properties, carry over. In particular, the following holds:
\begin{enumerate}[i)]
\item The mapping $\displaystyle t \mapsto \int_H \varphi(u)\dd\mu_t^{\varepsilon}(u)$ is continuous for every $\varphi \in C_b(H)$.
\item $\displaystyle\sup_{0\leq t\leq T}\int_H \|u\|_{L^2}^2\dd\mu^{\varepsilon}_t(u) \leq \int_H \|u_0\|_{L^2}^2\dd\overline{\mu}(u_0) \leq C \int_H \|u_0\|_{H^1}^2\dd\overline{\mu}(u_0)$.
\item $\displaystyle\int_0^T\int_H \|u\|_{H^1}^2\dd\mu_t^{\varepsilon}(u)\dt \leq T\int_H \|u_0\|_{H^1}^2\dd\overline{\mu}(u_0)$.
\item $\displaystyle\sup_{0\leq t \leq T}\int_{\lbrace \|u\|_{L^2} > r \rbrace} \|u\|_{L^2}^2\dd\mu_t^{\varepsilon}(u) \leq \int_{\lbrace \|u_0\|_{L^2} > r\rbrace} \|u_0\|_{L^2}^2\dd\overline{\mu}(u_0)$.
\item $\displaystyle-\int_0^T\int_H \theta'(t)\Phi(u) + (u\otimes u,\theta(t)\nabla\Phi'(u))_{L^2}\dd\mu^{\varepsilon}_t(u)\dt -\varepsilon \int_0^T\int_H (u,\theta(t)\Delta\Phi'(u))_{L^2} \dd\mu^{\varepsilon}_t(u)\dt = \int_H\theta(0)\Phi(u_0)\dd\overline{\mu}(u_0)$
for all $\Phi \in \CalT$ and $\theta \in C^1_c([0,T))$. 
\end{enumerate}
As ii) - iv) are uniform estimates in $\varepsilon$, we may indeed apply \Cref{thm: foias comp thm} to obtain a family of Borel probability measures $\lbrace \mu_t \rbrace_{0\leq t \leq T}$ representing a weak-* accumulation point in $L^1(0,T;C_2(H))'$. In particular iv) implies \eqref{eq: foias comp thm assump4}.
Note that all terms in v) are already in $L^1(0,T;C_b(H))$ or $L^1(0,T;C_2(H))$.
Therefore, one can conclude from \Cref{thm: foias comp thm} that $\lbrace \mu_t \rbrace_{0\leq t \leq T}$ satisfies \eqref{eq: Foias-Liouville equiv form}.\\
The energy inequality \eqref{eq: energy inequ construction} follows from the following considerations: Let $\varphi \in L^1(0,T)$ be non-negative. Due to ii) in \Cref{thm: foias comp thm}, for any $\eta > 0$, there exists $\varepsilon > 0$ such that 
\begin{equation}
\begin{split}
&\int_0^T\varphi(t)\int_H\|u\|_{L^2}^2\dd\mu_t(u)\dt = \int_0^T\int_H\varphi(t)\|u\|_{L^2}^2\dd\mu_t(u)\dt \leq \int_0^T\int_H \varphi(t)\|u\|_{L^2}^2\dd\mu_t^\varepsilon(u)\dt + \eta\\
= &\int_0^T\varphi(t)\int_H \|S_t^\varepsilon u_0\|_{L^2}^2\dd\overline{\mu}(u_0)\dt + \eta \leq \int_0^T\varphi(t)\int_H \|u_0\|_{L^2}^2\dd\overline{\mu}(u_0)\dt + \eta. 
\end{split}
\end{equation}
\end{proof}

Under a weak statistical scaling assumption, relating the second and third order longitudal structure functions, the 3D vanishing viscosity limit of statistical solutions was considered by Fjordholm, Mishra, Weber in \cite{FMW22} by employing a compactness result for correlation measures, first published in \cite{FLMW20} (see also \cite[Theorem 2.4]{LMP21} for a similar theorem). We will also state this compactness theorem in the following and roughly compare it to \Cref{thm: foias comp thm}.

We use the spaces $\CalH^{k,2}(\T^2;\R^2)$ and $\CalH^{k,2}_1(\T^2;\R^2)$ for $k \in \N$, as introduced prior to \Cref{def: stats sol moment based}.

\begin{theorem}\label{thm: corr meas comp thm}
Let $(\corrM_n)_{n\in\N} \subset \CalL^2([0,T],\T^2;\R^2)$ be a sequence of correlation measures satisfying
\begin{equation}\label{eq: corr meas comp thm finite l2 moment}
C_1 \coloneqq \sup_{n\in\N}\esssup_{0 \leq t \leq T} \int_{\T^2}\langle \nu^1_{n,t,x}, |\xi|^2\rangle\dx < \infty,
\end{equation}
\begin{equation}\label{eq: corr meas comp mod continuity}
\lim_{r \to 0}\limsup_{n\to\infty} \int_0^T \omega_r^p(\nu_{n,t}^2)\dt = 0
\end{equation}
with the integrand $\omega_r^p(\nu_{n,t}^2)$ defined as in \Cref{def: corr meas}. After passing to a subsequence, there exists $\corrM \in \CalL^2([0,T],\T^2;\R^2)$ such that
\begin{enumerate}[i)]
\item weak-* convergence $\bm{\nu}_{n} \overset{\ast}{\rightharpoonup} \bm{\nu}\,(n\to\infty)$ in $\CalH_0^{k*}([0,T],\T^2;\R^2)$ holds for every $k \in \N$,
\item $\corrM$ also satisfies \eqref{eq: corr meas comp thm finite l2 moment},
\item $\int_0^T \omega_r^p(\nu_t^2)\dt \leq \liminf_{n\to\infty}\int_0^T \omega_r^p(\nu_{n,t}^2)\dt$ for every $r > 0$,
\item for every $k \in \N$ and both non-negative $\varphi \in L^1([0,T] \times (\T^2)^k)$ and $\kappa \in C((\R^2)^k)$, letting $g(t,x,\xi) \coloneqq \varphi(t,x)\kappa(\xi)$, we have
\[\langle \nu^k,g\rangle \leq \liminf_{n \to \infty} \langle \nu_{n}^k,g\rangle.\] 
\item Moreover, if $(\corrM_n)_{n\in\N}$ has uniformly bounded support in the sense that for the associated Borel probability measures $(\lbrace \mu_t^n\rbrace_{0\leq t \leq T})_{n\in\N}$, given by \Cref{thm: corr measures main thm}, and some $M > 0$, $\|u\|_{L^2} \leq M$ holds for $\mu_t^n$-a.e.\ $u \in L^2(\T^2;\R^2)$, every $n \in \N$ and a.e.\ $0 \leq t \leq T$, then
\[\lim_{n\to\infty}\int_{(\T^2)^k}\bigg|\int_0^T \langle \nu^k_{n,t,x}, g(t,x)\rangle - \langle \nu^k_{t,x}, g(t,x)\rangle\dt\bigg|\dx = 0\]
for every $g \in \CalH^{k,2}_1([0,T],\T^2;\R^2)$.
\end{enumerate}
\end{theorem}   

We remark that the basic assumptions \eqref{eq: foias comp thm assump2}, \eqref{eq: foias comp thm assump3} in \Cref{thm: foias comp thm} are, likewise to \eqref{eq: corr meas comp thm finite l2 moment} and \eqref{eq: corr meas comp mod continuity} in \Cref{thm: corr meas comp thm}, used to be able to pass to the limit in an appropriate preliminary way. However, in order to have convergence of the integrals in both types of Liouville equations \eqref{eq: Foias-Liouville equiv form} and \eqref{eq: moment based Liouville}, e.g., when considering the vanishing viscosity limit, one needs better results on convergence along certain observables. In \Cref{thm: foias comp thm}, the space $C_2(H)$ is specifically tailored towards the quadratic term $u \mapsto (u \otimes u,\nabla\Phi'(u))$, $\Phi \in \CalT$ and thereby one needs to make sure that ii) in  \Cref{thm: foias comp thm} is also satisfied. Similarly, all terms in \eqref{eq: moment based Liouville} are in $\CalH^{k,2}_1$ and v) in \Cref{thm: corr meas comp thm} is somewhat crucial.

More in detail: The assumptions \eqref{eq: foias comp thm assump2} and \eqref{eq: corr meas comp thm finite l2 moment} on finiteness of the second $\|\cdot\|_{L^2}$ related moments are equivalent.\\
The second assumption \eqref{eq: foias comp thm assump3} in \Cref{thm: foias comp thm} is stronger than \eqref{eq: corr meas comp mod continuity}. Indeed, \eqref{eq: foias comp thm assump3} implies \eqref{eq: corr meas comp mod continuity} as the following computation shows. Let $(\corrM_n)_{n\in\N} \subset \CalL^2([0,T],\T^2;\R^2)$ be a sequence of correlation measures satisfying \eqref{eq: corr meas comp thm finite l2 moment} with associated families of Borel probability measures $(\lbrace \mu_t^n\rbrace_{0 \leq t \leq T})_{n\in\N}$ satisfying \eqref{eq: foias comp thm assump1} and \eqref{eq: foias comp thm assump2}. Suppose that $(\lbrace \mu_t^n\rbrace_{0 \leq t \leq T})_{n\in\N}$ also satisfies \eqref{eq: foias comp thm assump3}, that is $C_2 \coloneqq \sup_{n\in\N}\int_0^T\int_H \|u\|_{H^1}^2\dd\mu_t^n(u)\dt < \infty$. On $L^2(\T^2;\R^2)$, we use the Fourier representation $u(x) = \sum_{k \neq 0}\hat{u}_k e^{ik\cdot x}$ for $u \in H$ and $x \in \T^2$. Then, some computations, the identity \eqref{eq: corr meas main thm identity}, the Plancherel theorem and the estimate $|1-e^{i\beta}| \leq \tilde{C}|\beta|$ for every $\beta \in \R$ and some constant $\tilde{C} > 0$ yield
\begin{equation}
\begin{split}
\int_0^T \omega_r^p(\nu_{n,t}^2)\dt &= \int_0^T \int_{\T^2}\avint_{B_r(x)} \langle \nu^2_{n,t,x,y}, |\xi_1-\xi_2|^2 \rangle\,\dd y\dx\dt\\
&= \int_0^T\int_{L^2}\int_{\T^2}\avint_{B_r(0)}|u(x)-u(x+h)|^2\,\dd h\dx\dd\mu_t^n(u)\dt\\
&= \int_0^T\int_{L^2}\int_{\T^2}\avint_{B_r(0)}\bigg|\sum_{k \neq 0}\hat{u}_ke^{ik\cdot x}(1- e^{ik\cdot h})\bigg|^2\,\dd h\dx\dd\mu_t^n(u)\dt\\
&\leq \int_0^T\int_{L^2} \avint_{B_r(0)} \sum_{k\neq 0}|\hat{u}_k|^2|1-e^{ik\cdot h}|^2\,\dd h\dd\mu_t^n(u)\dt\\
&\leq \int_0^T\int_{L^2} \sum_{k\neq 0}|\hat{u}_k|^2 (\tilde{C}|k|r)^2\dd\mu_t^n(u)\dt\\
&\leq \tilde{C}^2 r^2 C_2,
\end{split}
\end{equation}
which converges uniformly in $n$ to 0 in the limit $(r \to 0)$ so that \eqref{eq: corr meas comp mod continuity} holds.\\
The converse is false, however, as \eqref{eq: corr meas comp mod continuity} does not necessarily make any implications on the mean $H^1$ norm. For instance, if we consider a bounded sequence $(u_n)_{n\in\N}$ in $L^2(\T^2;\R^2)$, then the associated (constant in time) atomic correlation measures can be shown to satisfy \eqref{eq: corr meas comp mod continuity} if and only if $(u_n)_{n\in\N}$ is precompact in $L^2(\T^2;\R^2)$.\\
\hfill\\
However, comparing the third assumption \eqref{eq: foias comp thm assump4} in \Cref{thm: foias comp thm} and v) in \Cref{thm: corr meas comp thm} respectively, the much stronger integrability assumptions that were already needed in \Cref{def: stats sol moment based} of moment-based statistical solutions also come in again. Indeed, it does not seem possible in the proof of \Cref{thm: corr meas comp thm}, given in \cite{FLMW20}, to replace the assumption of a uniformly bounded support in $L^2(\T^2;\R^2)$ with something much milder such as \eqref{eq: foias comp thm assump4}.\\
\hfill\\

\subsubsection{Discrete Approximations Based on the Krein-Milman Theorem}\label{II}
Generalizing the existence result by Chae in \cite{C91p}, two of the authors recently applied in \cite{WW23} the second, here to be discussed, method of constructing statistical solutions by using discrete approximations based on the Krein-Milman theorem. This method appears to have been first presented in \cite{FMRT01} for the Navier-Stokes equations and was later elevated to a much more general and abstract setting in \cite{BMR16}. In our setting of the two-dimensional incompressible Euler equations, we can formulate the main existence result as follows, which incorporates the right assumptions in order for the approximation to work. We give a sketch of the arguments afterwards.

\begin{theorem}\label{thm: Krein-Milman approx}
Let $X$ be a measurable subset of $H$ and let $\CalU \subset C([0,T];H_w)$ consist of weak solutions of the Euler equations. Let $\overline{\mu}$ be a Borel probability measure on $H$, concentrated on $X$ and suppose that
\begin{enumerate}[A)]
\item $\Pi_0\CalU = X$,
\end{enumerate}
there exists a family $\CalK'(X)$ of compact subsets of $X$ such that
\begin{enumerate}[A)]
\setcounter{enumi}{1}
\item $\overline{\mu}$ is tight with respect to $\CalK'(X)$, i.e.
\[\forall A \in \CalB(H): \overline{\mu}(A) = \sup_{\substack{K \in \CalK'(X) \\ K \subset A}} \overline{\mu}(K),\]
\item For every $K \in \CalK'(X)$, $\Pi_0^{-1}(K) \cap \CalU$ is compact in $C([0,T];H_w)$.
\end{enumerate}
Then there exists a trajectory statistical solution $\mu$
of the Euler equations with initial distribution ${\Pi_0}_\sharp\mu = \overline{\mu}$.
\end{theorem}

Condition A) just means that for every initial datum $u_0 \in X$, there exists a weak solution in $\CalU$ to the corresponding Cauchy problem. B) and C) mean that for a large enough family of compact sets of initial data, the weak solution trajectories in $\CalU$ with initial data in one of those sets are compact.\\ 
This theorem can roughly be proved as follows (cf. \cite{FMRT01,FRT13,BMR16}):
\begin{itemize}
\item Assumption B) allows one to reduce the proof to the case of $\overline{\mu}$ being carried by some $K \in \CalK'(X)$. Then we may view $\overline{\mu}$ as a measure on $K$.
\item The space $C(K)'$ along with the weak-* topology is a Hausdorff locally convex space and the unit ball $\mathbb{B}\ni \overline{\mu}$ is compact and convex.
\item Therefore, by the Krein-Milman theorem, $\mathbb{B}$ is equal to the closure of the convex hull of its extremal points, the latter being precisely the Dirac measures on $K$. In other words, for every $N \in \N$, there exist $n_N \in \N, \lambda_1^N,...,\lambda_{n_N}^N \geq 0$ and $u_{0,1}^N,...,u_{0,n_N}^N \in K$ such that
\[\sum_{i=1}^{n_N} \lambda_i^N = 1 \text{ and } \mu_{0,N} \coloneqq \sum_{i=1}^{n_N} \lambda_i^N\delta_{u_{0,i}^N} \overset{\ast}{\rightharpoonup} \overline{\mu}\,\quad (N \to \infty).\]
\item Now for each $N \in \N$, we define 
\[\mu_N \coloneqq \sum_{i=1}^{n_N} \lambda_i^N\delta_{u_i^N}\]
as Borel probability measures on $\kappa \coloneqq \Pi_0^{-1}(K) \cap \CalU$, where for each $N \in \N$ and $i = 1,...,n_N$, we let $u_i^N \in \kappa$ such that $u_i^N(0) = u_{0,i}^N$. This is possible due to A).
\item As $\kappa$ is compact, $C(\kappa)$ is separable and by the Banach-Alaoglu theorem, after passing to a subsequence, $(\mu_N)_{N\in\N}$ can be seen to converge weakly-* in $C(\kappa)'$ to some $\mu$, which, extended to a measure on $C([0,T];H_w)$, can be shown to be a trajectory statistical solution with ${\Pi_0}_{\sharp}\mu = \overline{\mu}$.
\item A corresponding phase space statistical solution can be obtained from \Cref{thm: traj stats sol -> phase space stats sol}.
\end{itemize}

It is apparent from this sketch that this theorem also holds in the better behaved setting of \Cref{thm: general existence thm stats sols}, when $H_w$ is replaced with $H$ or $H^1$.

\begin{proof}[Proof of \Cref{thm: general existence thm stats sols} (Sketch 2)]
In the setting of \Cref{thm: Krein-Milman approx}, let $X = H^1$ and $\CalU$ be the set of all weak solutions 
as in \Cref{prop: existence euler}, so that A) is immediately satisfied.
For properties B) and C), choose $\CalK'(X)$ to be the set of all subsets of $H^1$, compact w.r.t. $\|\cdot\|_{H^1}$.\\
Due to \eqref{eq: general ex thm, bdd mean enstrophy}, $\overline{\mu}$ is concentrated on $H^1$. Moreover, $\|\cdot\|_{H^1}$ and $\|\cdot\|_{L^2}$ generate the same Borel-$\sigma$-algebra on $H^1$ so that we can see $\overline{\mu}$ as a measure on $H^1$. Then B) follows from the general fact of Borel probability measures on Polish spaces being tight ~\cite[Theorem 3.2]{P67}.

Moving on to C), for every $K \in \CalK'(X)$, $\kappa \coloneqq \Pi_0^{-1}(K) \cap \CalU$ is in fact already compact in $C([0,T];H^1)$: $K$ being bounded in $H^1$ yields uniform bounds on the quantities involved in i) - iii) in \Cref{prop: existence euler} which in term implies uniform boundedness of $\kappa$ in $H^1$. Then the precompactness in $C([0,T];H)$ follows from iii) and the Aubin-Lions lemma.

Convergence in $C([0,T];H)$ is enough to conclude that limits of subsequences of velocities in $\CalU$ are once again weak solutions of the Euler equations and i) also holds in the limit. Similarly, from weak lower semicontinuity of the left-hand side in iii) and the strong convergence of the right-hand side, we may see that such limits of sequences in $\kappa$ also satisfy iii). 

As for property ii), due to compactness of $K$ in $H^1$, we may assume strong compactness of the initial vorticities in $\kappa$. We also obtain from ii) weak-* compactness of the vorticities of $\kappa$ in $L^\infty(0,T;L^2(\T^2))$ and combined with the compactness of the velocities in $C([0,T];H)$, this suffices to show that the vorticity of such a limit in $\kappa$ is a weak solution of the vorticity formulation \eqref{eq: vort equ} of the Euler equations. Now it suffices to note that weak vorticity solutions in $L^2(\T^2)$ are automatically renormalized and in fact, strong convergence of the vorticities holds in $C([0,T];L^2(\T^2))$ in the inviscid limit cf. \cite{LMN05,CCS21,NSW21}. Note that by a Calder{\'o}n-Zygmund argument, continuity and compactness in $L^2(\T^2)$ of the vorticity is enough to obtain continuity and compactness of the full gradient in $L^2(\T^2)$. 

In summary, $\kappa$ is compact in $C([0,T];H^1)$ and \Cref{thm: Krein-Milman approx} is applicable so that we obtain a trajectory statistical solution $\mu$ on $\CalU \subset C([0,T];H^1)$ with initial distribution $\overline{\mu}$. In combination with \Cref{thm: traj stats sol -> phase space stats sol}, one obtains the phase space statistical solutions $\lbrace \mu_t \rbrace_{0\leq t \leq T} = \lbrace {\Pi_t}_\sharp\mu\rbrace_{0\leq t \leq T}$ with initial distribution $\overline{\mu}$. In particular, the energy inequality is satisfied due to
\begin{align*}
\int_{H^1} \|u\|_{L^2}^2\dd\mu_t(u) &= \int_{C([0,T];H^1)}\|u(t)\|_{L^2}^2\dd\mu(u)\\
& \leq \int_{C([0,T];H^1)}\|u(0)\|^2_{L^2}\dd\mu(u) = \int_{C([0,T];H^1)}\|u_0\|_{L^2}^2\dd\overline{\mu}(u_0).
\end{align*}    
\end{proof}

\begin{remark}\label{rem: enstrophy balance}
The same argument in Sketch 2 that showed the energy inequality can also be used to show the conservation of enstrophy, i.e. the phase space statistical solution satisfies
\begin{equation}\label{eq: enstrophy balance statistical}
\int_{H^1} \|\omega(u)(t)\|_{L^2}\dd\mu_t(u) = \int_{H^1} \|\omega(u_0)\|_{L^2}\dd\overline{\mu}(u_0).
\end{equation}
\end{remark}

\subsubsection{Measurable Selections and Push-Forward Constructions}
The third method that we very briefly review here is the construction of pushforward measures of initial distributions along a measurable selection operator of the possibly non-unique weak solutions. This is perhaps the most straightforward way in which one can construct statistical solutions and ties in with the derivation of the notion of statistical solution, as discussed in the introduction, where we also assumed the existence of a solution operator.

This method was employed by Ladyzhenskaya and Vershik in \cite{LV78} to solve the Hopf statistical equation in coordinate form.\\
More precisely, they applied a version of the now classical Kuratowski and Ryll-Nardzewski selection principle to the three-dimensional Navier-Stokes equations. Here, we would like to state a more recent and refined result due to Cardona and Kapitanski on semiflow selections in~\cite{CK20}, where also the example of the Navier-Stokes equations is discussed. Their work has also inspired the recent articles by Breit, Feireisl and Hofmanov{\'a} on semiflow selection for the isentropic and complete Euler system \cite{BFH19,BFH20} and this selection and push-forward principle has then also been used by Fanelli and Feireisl to show existence of a newly introduced concept of statistical solutions to the barotropic Navier-Stokes system \cite{FF20}.

For every $s \geq 0$, the splicing at time $s$ of two functions $v,w$ on $[0,\infty)$ with $w(s) = v(0)$ is defined as 
\[(w \bowtie_s v)(t) = \begin{cases}
w(t), & 0 \leq t \leq s\\
v(t-s), & t \geq s
\end{cases}.\]
We also define the time shifts $\theta_s$ for every $s \geq 0$ so that $\theta_sw(t) = w(t+s)$ for every function $w$ on $[0,\infty)$ and every $t \in [0,\infty)$.

\begin{theorem}\label{thm: meas semiflow selection}
Let $X$ be a Polish space, $C_{loc}([0,\infty);X)$ be the space of continuous functions endowed with the compact-open topology and for $x \in X$, denote by $\Omega_x$ the subset of trajectories in $C_{loc}([0,\infty);X)$ starting at $x$. Suppose that $\Psi$ is a mapping from $X$ onto the non-empty subsets of $C_{loc}([0,\infty);X)$ satisfying the following properties:
\begin{enumerate}[a)]
\item For every $x \in X$, $\Psi(x)$ is a compact subset in $\Omega(x)$.
\item $\Psi$ is measurable in the sense that for every closed set $A \subset \Omega$, $\lbrace x \in X : \Psi(x) \cap A \neq \emptyset\rbrace$ is a Borel measurable subset of $X$.
\item $\Psi$ is compatible with the semigroup $\theta_s$ in the sense that for every $s \geq 0, x \in X$, we have
\[w \in \Psi(x) \Rightarrow \theta_s(w)\in \Psi(w(s)).\]
\item $\Psi$ satisfies a glueing property in the sense that for $w \in \Psi(x)$ and $v \in \Psi(w(s))$ for some $s \in [0,\infty)$, also $w \bowtie_s v \in \Psi(x)$. 
\end{enumerate}
Then there exists a measurable semiflow selection $\psi$ of $\Psi$, i.e., $\psi$ is a Borel measurable mapping from $X$ to $C_{loc}([0,\infty);X)$ so that $\psi(x) \in \Psi(x)$ for every $x \in X$ and it has the semiflow property $\psi(x,t+s) = \psi(\psi(x,t),s)$ for all $t,s \geq 0$.
\end{theorem}

\begin{proof}[Proof of \Cref{thm: general existence thm stats sols} (Sketch 3)]
We consider $X = H^1$ and for any $u_0 \in H^1$, we let $\Psi(u_0) \subset C_{loc}([0,\infty);H^1)$ be the set of all weak solutions with properties i) - iii) in \Cref{prop: existence euler} and initial datum $u_0$.

The assumptions a)-d) in \Cref{thm: meas semiflow selection} are satisfied for similar or identical reasons as in our second sketch of the proof of \Cref{thm: general existence thm stats sols}, more precisely:

$\Psi(u_0)$ is non-empty for every $u_0 \in H^1$ by \Cref{prop: existence euler}. Moreover, $\Psi(u_0)$ is compact in $C_{loc}([0,\infty);H^1)$, as we already saw in Sketch 2 for finite time intervals and from which we may obtain the global case by a diagonal argument. 

As for b), for any closed set $A \subset C_{loc}([0,\infty);H^1)$, one can also prove that $\lbrace u_0 \in H^1 : \Psi(u_0) \cap A \neq \emptyset \rbrace$ is closed in $H^1$ by employing the same arguments.

For c) and d), all of the properties of weak solutions with the above properties are preserved for time-shifts and splicings, as can be seen rather immediately, perhaps with exception of the weak velocity and renormalized vorticity formulation of the Euler equations. This last point can be proved by approximating a given test function appropriately as for instance in \cite[Section 4.3]{CK20} in the case of the Navier-Stokes equations.

In summary, \Cref{thm: meas semiflow selection} is applicable and we may consider a measurable selection $S$ of $\Psi$. Due to its measurability, we may consider the pushforward measure $\mu \coloneqq S_\sharp\overline{\mu}$ on $C_{loc}([0,\infty);H^1)$ so that $\mu$ is a trajectory statistical solution with initial distribution $\overline{\mu}$. Using \Cref{thm: traj stats sol -> phase space stats sol}, one obtains a phase-space statistical solution as desired in \Cref{thm: general existence thm stats sols} also satisfying the energy inequality \eqref{eq: energy inequ construction} everywhere, as can be shown as in our second sketch.
\end{proof}

\begin{remark}
Similarly to \Cref{rem: enstrophy balance}, the statistical solution constructed here also satisfies the enstrophy balance \eqref{eq: enstrophy balance statistical}.
\end{remark}

\subsubsection{Conclusions}
We close this subsection by remarking that all three methods that we described here have in common that they actually rely in some form on compactness properties of weak deterministic solutions and/or approximations thereof: For instance in the two-dimensional case when considering the vanishing viscosity limit for initial data with vorticity in $L^2$, the uniform bounds \eqref{eq: foias comp thm assump2} and \eqref{eq: foias comp thm assump3} in \Cref{thm: foias comp thm} are the analogues of the uniform estimates which in the deterministic case allow one to pass to the limit, and are therefore also satisfied by the pushforward approximations along the two-dimensional Navier-Stokes solution operators.

In Subsection~\ref{II}, the compactness of weak solutions with initial data in certain compact sets is precisely assumption C) in \Cref{thm: Krein-Milman approx}.
Similarly, also in \Cref{thm: meas semiflow selection} the compactness of solution trajectories sharing a common initial datum is assumed, though we mention here that for the classical Kuratowski and Ryll-Nardzewski selection theorem, closedness would be enough.

Overall, this close connection to existence and compactness of weak deterministic solutions yields at least that the constructed statistical solutions typically satisfy similar properties.

While for any $u_0 \in H$, infinitely many weak solutions of the incompressible Euler equations, possibly even satisfying the energy inequality, may be constructed \cite{DLS08,W11}, the lack of compactness of these convex integration solutions makes it impossible at the moment to construct statistical solutions of the incompressible Euler equations with the presented tools for any initial distribution $\overline{\mu}$ on $H$ without any further assumptions such as \eqref{eq: general ex thm, bdd mean enstrophy}.

\subsection{Comparison to Measure-Valued Solutions}\label{Comparison to measure valued solutions}

\subsubsection{Statistics of Singular Limits}
The formulation of statistical solutions in terms of correlation measures allows for an alternative, somewhat simpler interpretation by considering all infinitely many multipoint correlations instead of a single distribution on an infinite dimensional space. In contrast, measure-valued solutions only describe the one-point statistics of a system.

One would hope that by ``adding'' (all) multipoint correlations to Young measures and measure-valued solutions, results on uniqueness and stability would be easier to achieve. This, however, is different from the context in which DiPerna and Majda introduced the notion of measure-valued solution in \cite{DM87osc} to treat the inviscid limit problem, cf.~\Cref{sec: mvs}. Correlation measures as defined in \Cref{subsec: definitions} appear to be too strong of a concept to study oscillatory limits.

As an easy example, consider the function $u\colon \T \to \R$, given by
\[ u(x) = 
\begin{cases}
0, & 0 < x \leq \frac{1}{2}\\
1, & \frac{1}{2} < x \leq 1
\end{cases}
\]
and extended periodically. Then we define $u_n(x) \coloneqq u(nx)$ for all $x \in \T$ and $n \in \N$. The two-point correlations of $(u_n)_{n\in\N}$ in the limit are described by the Young measure $\lbrace \nu_{x,y}^2 \rbrace_{(x,y) \in \T^2}$ generated by $v_n(x,y) \coloneqq (u_n(x),u_n(y)), n \in \N, (x,y) \in \T^2$. It is not hard to see that, constantly in $(x,y) \in \T^2$,
\[\nu_{x,y}^2 \equiv \frac{1}{4}\delta_{(0,0)} + \frac{1}{4}\delta_{(0,1)} + \frac{1}{4}\delta_{(1,0)} + \frac{1}{4}\delta_{(1,1)}.\]
However, this constant Young measure cannot be part of a correlation measure as the diagonal continuity \eqref{eq: diag cont} is violated since we have, constantly in $r > 0$,
\[\int_{\T^2}\avint_{B_r(x)} \langle \nu_{x,y}^2, |\xi_1-\xi_2|^2\rangle\dd y\dx = \frac{1}{2} > 0.\]
On a related note, it is possible, however, to associate Borel probability measures to (generalized) Young measures and also to generalize the fundamental theorem of Young measures to obtain a notion of generating Young measures by sequences of Borel probability measures, cf.~\cite{C91g}.

\subsubsection{Weak-Strong Uniqueness}
We mentioned before that one would hope to obtain better stability and uniqueness results for statistical solutions compared to measure-valued solutions. In this regard, much is yet to be desired.

If the initial data in the support of an initial distribution yields unique solutions in a certain class, it is usually quite simple to show uniqueness of trajectory statistical solutions, see for instance \cite[Theorem 2.6]{WW23} for the case of the two-dimensional Euler equations with initial data and solutions in the Yudovich class or, more generally,~\cite[Chapter IV §6]{VF88}.

For phase space statistical solutions, similar uniqueness results are much harder to achieve. In case of the two-dimensional Navier-Stokes equations, see for instance Theorem 1.2 in Chapter V of \cite{FMRT01}. In \cite{LMP21}, a weaker kind of weak-strong uniqueness was shown, building upon weak-strong uniqueness of measure-valued solutions of the Euler equations, cf.~\Cref{thm: weak-strong mvs}. As it fits nicely to the content of this article and for the convenience of the reader, we present it here but only sketch the proof of the main result. The following holds in both $d = 2$ and $d = 3$ dimensions.\\
By \textit{strong} in this context of weak-strong uniqueness, we mean that we consider statistical solutions with an initial distribution that is concentrated on a (slightly larger) set of initial data that yields classical solutions on the considered time interval.

More precisely, we consider the set $\CalC \subset C^1(\T^d;\R^d) \cap H$ of initial data $\overline{v}$ which yields classical solutions of the Euler equations $v \in C^1(\T^2 \times [0,T];\R^2)$ so that $C(\overline{v}) \coloneqq \sup_{t \in [0,T]} \|\nabla v(t)\|_{\infty} < \infty$. Now, for every $n \in \N$, let 
\[\CalG_n = \bigcup_{\overline{v} \in \CalC} B^{H}_{e^{-C(\overline{v})T}/n}(\overline{v}) \text{ and define } \CalG = \bigcap_{n\in\N}\CalG_n.\] 
To be able to build on the weak-strong uniqueness of measure-valued solutions, we need to use the notion of dissipative statistical solutions. Before giving the precise definition, let us introduce the following notation: For a Borel probability measure $\mu$ on $H$ as in \Cref{thm: corr measures main thm}, we denote by $\nu^1(\mu) = \lbrace \nu_x^1(\mu)\rbrace_{x \in \T^d}$ the first Young measure in the hierarchy of Young measures that constitute the correlation measure associated to $\mu$. Also, for such measures $\mu$ and $\alpha = (\alpha_1,...,\alpha_N) \in [0,1]$ such that $\sum_{i=1}^N \alpha_i = 1$, we consider the set $\Lambda(\alpha,\mu) \coloneqq \lbrace (\mu_1,...,\mu_N) : \mu = \sum_{i=1}^N \alpha_i \mu_i\rbrace$ of $N$-tuples of Borel probability measures on $H$. Note that always $(\mu,...,\mu) \in \Lambda(\alpha,\mu)$ so that this set is never empty.

\begin{definition}\label{def: dissipative stats sol}
A statistical solution $\lbrace \mu_t \rbrace_{0 \leq t < T}$ in phase space of the Euler equations with initial distribution $\overline{\mu}$ is called {\em dissipative statistical solution} of the Euler equations if for every $N \in \N$, $\alpha \in [0,1]^N$ satisfying $\sum_{i=1}^N \alpha_i = 1$ and $(\overline{\mu}_1,...,\overline{\mu}_N) \in \Lambda(\alpha,\overline{\mu})$, there exist $(\hat{\mu}_{1,t},...,\hat{\mu}_{N,t}) \in \Lambda(\alpha,\mu_t)$ for almost every $0 \leq t \leq T$ such that for each $i = 1,...,N$, $\lbrace \nu^1(\hat\mu_{i,t}) \rbrace_{0 \leq t \leq T}$ is an oscillation measure-valued solution of the incompressible Euler equations with initial Young measure $\nu^1(\overline{\mu}_i)$. 
\end{definition}

\begin{theorem}[\cite{LMP21}]\label{thm: weak-strong unique stats sols}
Let $\overline{\mu}$ be a Borel probability measure on $H$ which is concentrated on $\CalG$ and has bounded support in $B_M^H$ for some $M > 0$. Then there exists a unique dissipative statistical solution of the incompressible Euler equations with initial distribution $\overline{\mu}$.
\end{theorem}

\begin{remark}
\begin{enumerate}[i)]
\item The definition in \cite{LMP21} of dissipative statistical solutions also contains a property of (weak) time-regularity. This, however, is not required for \Cref{thm: weak-strong unique stats sols} and is consequently omitted in \Cref{def: dissipative stats sol}.
\item \Cref{thm: weak-strong unique stats sols} in particular yields existence and uniqueness of dissipative statistical solutions in the two-dimensional case for initial distributions concentrated on $C^{1,\beta}(\T^2;\R^2)\cap H$ for some $\beta \in (0,1)$ and of bounded support.
\end{enumerate}
\end{remark}

\begin{proof}[Proof of \Cref{thm: weak-strong unique stats sols} (sketch)]
\textit{Partitions and approximations:} We begin by introducing partitions of $\CalG$.\\
Since $H$ is separable, countably many open balls in the definition of $\CalG_n$ suffice to cover $\CalG_n$ for every $n \in \N$. By considering the countable union over $n \in \N$ of the countably many centers of balls that cover $\CalG_n$, we obtain a countable set $\{ \overline{u}_i\}_{i\in I}$ (for simplicity assume $I = \N$) such that $\CalG_n = \bigcup_{i=1}^\infty B_{r_i/n}(\overline{u}_i)$ for every $n \in \N$, where $r_i \coloneqq e^{-C(\overline{u}_i)T}$ for every $i \in \N$. Then we partition $\CalG$ as follows: Let
\begin{align*}
\left\{
\begin{aligned}
S_1^{(n)} &= B_{r_i/n}(\overline{u}_i) \cap \CalG\\
\Sigma_1^{(n)} &= S_1^{(n)}\\
\alpha_i^{(n)} &= \overline{\mu}(S_1^{(n)})
\end{aligned}\right.\qquad
\left\{\begin{aligned}
S_{i+1}^{(n)} &= B_{r_{i+1}/n}(\overline{u}_{i+1})\setminus\Sigma_i^{(n)}\\
\Sigma_{i+1}^{(n)} &= \Sigma_i^{(n)} \cup S_{i+1}^{(n)}\\
\alpha_{i+1}^{(n)} &= \overline{\mu}(S_{i+1}^{(n)}).
\end{aligned}\right.
\end{align*}
Then $(S_i^{(n)})_{i=1}^\infty$ is a disjoint partition of $\CalG$ and $\sum_{i=1}^\infty \alpha_i^{(n)} = 1$.

\textit{Existence:} For all $i,n \in \N$, define $\overline{\mu}_i^{(n)} \coloneqq \overline{\mu}_i(\cdot \cap S_i^{(n)})$ so that $\overline{\mu} = \sum_{i=1}^\infty \alpha_i^{(n)}\overline{\mu}_i^{(n)}$ for every $n \in\N$. As $\overline{\mu}_i^{(n)} \approx \delta_{\overline{u}_i}$ for large $n$, we consider the discrete approximations
\[\overline{\mu}^{(n)} \coloneqq \sum_{i=1}^\infty \alpha_i^{(n)}\delta_{\overline{u}_i}, n \in \N,\]
of $\overline{\mu}$.

For every $i \in \N$, let $u_i \in C^1(\T^2 \times [0,T];\R^2)$ be the classical solution with initial data $\overline{u}_i$. Then we consider the time parametrized measures 
\[
\mu_t^{(n)} \coloneqq \sum_{i=1}^\infty \alpha_i^{(n)}\delta_{u_i(t)}, n \in \N, t \in [0,T].
\]
Then $\{\mu_t^{(n)}\}_{0 \leq t \leq T}$ can be seen to be a dissipative statistical solution with initial distribution $\overline{\mu}^{(n)}$ for every $n \in \N$.

Based on stability of classical solutions, which also motivated the specific choice of radii $r_i, i \in \N$, one can show that $(\{\mu_t^{(n)}\}_{0\leq t \leq T})_{n\in\N}$ is uniformly Cauchy with respect to the Wasserstein  distance $d_2$. Due to completeness of the Wasserstein space, one obtains measures $\lbrace \mu_t \rbrace_{0 \leq t \leq T}$ in the uniform limit which also form a dissipative statistical solution with initial distribution $\overline{\mu}$.

\textit{Uniqueness:} For the proof of uniqueness, suppose that $\lbrace \tilde{\mu}_t\rbrace_{0 \leq t \leq T}$ is also a dissipative statistical solution with initial distribution $\overline{\mu}$. Due to Step 2, it suffices to show that $\lim_{n \to \infty} d_2(\tilde{\mu}_t, \mu_t^{(n)}) = 0$ for all $0 \leq t \leq T$.

We now fix $n \in \N$. For large $N \in \N$, approximate $\mu_t^{(n)}$ by 
\[\tilde{\mu}_t^{(N,n)} \coloneqq \sum_{i=1}^N \alpha_i^{(n)}\delta_{u_i(t)} + \alpha_0^{(N,n)}\delta_{\CalO(t)},\]
where $\CalO$ denotes the constant zero function and $\alpha_0^{(N,n)} \coloneqq \sum_{i > N} \alpha_i^{(n)}$ so that $\tilde{\mu}_t^{(N,n)}$ is a probability measure. If $\alpha_0^{(N,n)}$ is $0$, then the zeroth terms can be omitted in the following. We then have $\alpha_0^{(N,n)} + \sum_{i=1}^N \alpha_i^{(n)} = 1$ and by letting $\overline{\mu}_0^{(N,n)} \coloneqq \frac{1}{\alpha_0^{(N,n)}}\sum_{i > N}\alpha_i^{(n)}\overline{\mu}_i^{(n)},$ it follows that 
\[(\overline{\mu}^{(N,n)}_0,\overline{\mu}^{(n)}_1,...,\overline{\mu}^{(n)}_N) \in \Lambda(\alpha,\overline{\mu}).\]
The definition of dissipative statistical solutions implies the existence of partitions
\[(\hat{\mu}_{0,t},...,\hat{\mu}_{N,t}) \in \Lambda(\alpha,\tilde{\mu}_t)\]
of $\tilde{\mu}_t$ such that for each $i = 1,...,N$, $\{\nu^1(\hat{\mu}_{i,t})\}_{0 \leq t < T}$ is a measure-valued solution starting at $\nu^1(\overline{\mu}_i^{(n)})$. Then the analogue of \Cref{thm: weak-strong mvs} in the incompressible case yields
\begin{equation}
\begin{split}
&\int_H \|u - u_i(t)\|_{L^2}^2\dd\hat{\mu}_{i,t}(u) = \int_{\T^2} \langle \nu_{x}^1(\hat{\mu}_{i,t}),|\xi - u_i(t)|^2\rangle\dx\\
\leq& e^{C(\overline{u}_i)T}\int_{\T^2}\langle\nu_x^1(\overline{\mu}_i),|\xi - \overline{u}_i|^2\rangle\dx = \frac{1}{r_i}\int_H \|u - \overline{u}_i\|_{L^2}^2\dd\overline{\mu}_i(u)
\end{split}
\end{equation}
for a.e. $t \in [0,T]$ and $i = 1,...,N$.

The left-hand side can be rewritten using the transport plan $\pi_t \coloneqq \hat{\mu}_{i,t} \otimes \delta_{u_i(t)}$ as $\int_{H \times H} |\xi_1 - \xi_2|^2\dd\pi_t(\xi_1,\xi_2)$ so that by definition of $d_2$, for every $i = 1,...,N$
\begin{equation}\label{eq: weak-strong unique wasserstein}
d_2^2(\hat{\mu}_{i,t},\delta_{u_i(t)}) \leq \frac{1}{r_i}\int_H\|u - \overline{u}_i\|_{L^2}^2\dd\overline{\mu}_i(u) \leq \frac{1}{r_i} \left(\frac{r_i}{n}\right)^2 \leq \frac{1}{n^2},
\end{equation}
where we used that $\overline{\mu}_i^{(n)}$ is concentrated on $S_i^{(n)} \subset B_{r_i/n}(\overline{u}_i)$. Also, 
\[d_2^2(\hat{\mu}_{0,t},\delta_{\CalO(t)}) \leq \int_H \|u\|_{L^2}^2\dd\hat{\mu}_{0,t}(u) \leq \int_H \|u\|_{L^2}^2\dd\overline{\mu}_{0}^{(N,n)}(u) \leq M^2.\]
Now, let $\varepsilon > 0$. As both \eqref{eq: weak-strong unique wasserstein} and $\lim_{n \to \infty} d_2(\mu_t^{(n)},\mu_t)=0$ are uniform in $N$ and $t$, we may choose $n \in \N$ large enough such that for every $N \in \N$ and $i=1,...,N$, $0 \leq t \leq T$,
\[
d_2^2(\hat{\mu}_{i,t},\delta_{u_i(t)}) \leq \frac{\varepsilon^2}{18} \text{ and } d_2(\mu_t,\mu_t^{(n)}) \leq \frac{\varepsilon}{3}.
\]
Then choose $N$ large enough such that 
\[\alpha_0^{(N,n)} = \sum_{i > N} \alpha_i^{(n)} < \frac{\varepsilon^2}{18 M^2}.\]
Consequently,
\begin{align*}
d_2^2(\tilde{\mu}_t,\tilde{\mu}_t^{(N,n)}) &\leq \alpha_0^{(N,n)}d_2^2(\hat{\mu}_{0,t},\delta_{\CalO(t)}) + \sum_{i=1}^N\alpha_i^{(n)}d_2^2(\hat{\mu}_{i,t},\delta_{u_i(t)})\\
&\leq \frac{\varepsilon^2}{18 M^2}M^2 + \frac{\varepsilon^2}{18}\sum_{i=1}^N \alpha_i^{(n)} \leq \frac{\varepsilon^2}{9}
\end{align*}
and likewise
\[
d_2^2(\tilde{\mu}_t^{(N,n)},\mu_t^{(n)}) \leq \sum_{i > N}\alpha_i^{(n)}d_2^2(\delta_{\CalO(t)},\delta_{u_i(t)}) \leq M^2\sum_{i > N}\alpha_i^{(n)} = M^2 \alpha_0^{(N,n)} \leq \frac{\varepsilon^2}{9}.
\]
Finally
\begin{equation*}
d_2(\tilde{\mu}_t,\mu_t) \leq d_2(\tilde{\mu}_t,\tilde{\mu}_t^{(N,n)}) + d_2(\tilde{\mu}_t^{(N,n)},\mu_t^{(n)}) + d_2(\mu_t^{(n)},\mu_t)
\leq \frac{\varepsilon}{3} + \frac{\varepsilon}{3} + \frac{\varepsilon}{3}
= \varepsilon.
\end{equation*} 
\end{proof}

\section{Discussion and Open Problems}\label{Discussion and open problems}
The following problems that we list here are not specifically related to the incompressible Euler or Navier-Stokes equations. Instead, they are problems on a more general level and some more insight and further examples would benefit the area.

\subsection*{Existence Independently of the Deterministic Problem}
The Dirac measures of weak solutions satisfy the Foia\cs-Liouville equation \eqref{eq: Foias-Liouville}. For non-unique solutions sharing the same initial data, convex combinations of the corresponding solutions also satisfy \eqref{eq: Foias-Liouville}.  In this way, the concept of statistical solutions allows for deterministic initial data which ``splits'' into non-deterministic distributions so that statistical solutions are a more flexible solution concept than weak solutions.

However, as mentioned in our conclusion of \Cref{subsec: construction stats sols}, current methods for obtaining existence of statistical solutions generally rely on existence and even some kind of compactness of the underlying deterministic solutions. It would be interesting to see further approaches primarily based on the Foia\cs-Liouville equation, perhaps even for equations and/or data where weak deterministic existence is unknown.

\subsection*{Connections Between Phase Space Statistical Solutions and Trajectory Statistical Solutions}
We have seen in \Cref{thm: traj stats sol -> phase space stats sol} that by projecting a trajectory statistical solution of the incompressible Euler equations at every point in time yields a phase space statistical solution. This also holds true in a very general setting \cite[Theorem 3.3]{BMR16}.

It is generally unknown if the converse holds as well. It is known for instance for the two-dimensional Navier-Stokes equations (with initial distribution of bounded support) as both the trajectory and phase space statistical solutions are uniquely given as pushforward measures of the initial distribution along the solution operator in trajectory or phase space respectively. Perhaps more interestingly, this is also true for time-average stationary statistical solutions of the three-dimensional Navier-Stokes equations \cite[Theorem 5.5]{FRT19}. So far, no counterexamples, i.e., phase space statistical solutions that cannot be obtained as projections of trajectory statistical solutions, can be found in the literature for any equation. This is related to the previous problem as it requires a further study of the Foia\cs-Liouville equation and its connection to the deterministic formulation.


\subsection*{Connections Between Young Measures and Correlation Measures}
Let us consider the map $\mathscr{F}\colon \CalL^2(\T^2;\R^2) \to L^\infty_{\operatorname{w}}(\T^2,\mathcal{P}(\R^2))$, where $\mathscr{F}$ projects onto the first Young measure of finite kinetic energy in the hierarchy of Young measures that constitute a correlation measure in $\CalL^2(\T^2;\R^2)$.

Then it is unclear to us if $\mathscr{F}$ is one-to-one or onto. Being onto would mean that every Young measure $\lbrace \nu_x \rbrace_{x \in \T^2}$ of finite kinetic energy would be the first correlation marginal of a correlation measure in $\CalL^2(\T^2;\R^2)$. Due to \Cref{thm: corr measures main thm}, this is equivalent to the existence of a Borel probability measure $\mu$ on $L^2(\T^2;\R^2)$ having finite second moment so that for every $h \in \CalH_0^1(\T^2;\R^2)$
\[\int_{L^2}\int_{\T^2} h(x,u(x))\dx\dd\mu(u) = \int_{\T^2} \langle \nu_x, h(x,\xi)\rangle\dx.\]
$\mathscr{F}$ being one-to-one would mean that such a $\mu$, if it exists, is unique.

Intuitively, $\mathscr{F}$ being onto while not being one-to-one seems reasonable. The latter, that is being one-to-one, would mean that the first correlation marginal determines all further multipoint correlations.

However, since the required properties of a correlation measure, in particular the diagonal continuity, are quite restrictive, this is not entirely obvious and proofs or explicit examples would be desirable. The only explicit example in this context is again that of a single Dirac Young measure which can only in a unique way, due to the diagonal continuity, be ``extended'' to a full correlation measure, namely the atomic correlation measure.

Finally, we would like to point here at the general problem of a lack of more or less explicit examples of statistical solutions. Unless the system is well-posed and all phase space statistical solutions are given as pushforward measures of an initial distribution along the solution semigroup, the only explicit examples at hand seem to be convex combinations of Diracs corresponding to weak solutions. Even the well-posed case is inconvenient to handle in the context of correlation measures and moment-based statistical solutions. For example, consider the case of a measurable semigroup $\lbrace S_t \rbrace_{t \geq 0}$ and a Borel probability measure $\mu_0$ on $L^2(\T^2;\R^2)$ having finite second moment so that we may define $\rho_t \coloneqq {S_t}_{\sharp}\mu_0$ for every $t \geq 0$. For any $h \in \CalH^k_0(\T^2;\R^2)$ the right-hand side in \eqref{eq: corr meas main thm identity} is
\[\int_{L^2}\int_{(\T^2)^k} h(x,u(x))\dx\dd\rho_t(u) = \int_{L^2}\int_{(\T^2)^k} h(x,({S_t}u_0)(x))\dx\dd\mu_0(u_0).\]
So it can be expressed in terms of the initial distribution and the given semigroup $\lbrace {S_t} \rbrace_{t \geq 0}$.\\ 
However, one cannot make sense of the left-hand side in \eqref{eq: corr meas main thm identity} similarly just in terms of the initial correlation measure corresponding to $\mu_0$ and $\lbrace {S_t} \rbrace_{t \geq 0}$. 
Note that here, for given $x \in (\T^2)^k$ and $t \geq 0$, $h(x,\cdot)\colon (\R^2)^k \to \R$ while ${S_t}\colon L^2(\T^2;\R^2) \to L^2(\T^2;\R^2)$, hence the composition $\xi \mapsto h(x,{S_t}\xi)$ does not make sense.


\end{document}